\documentclass[12pt]{article}
 \usepackage{amsmath, amsthm, amssymb, bbm, setspace,bigints}
  \usepackage{algorithm2e}
\RestyleAlgo{boxruled}
\LinesNumbered
 \usepackage[margin=1 in]{geometry}
\usepackage{graphicx}				
\usepackage{amssymb,color}					
\usepackage{amsmath}
\usepackage{epstopdf}
\usepackage{bbm}
\usepackage{mathtools}
\usepackage{authblk}
\usepackage{mathrsfs}

\usepackage{graphicx,array,booktabs,bbm,multicol,colortbl,tabularx}
\usepackage{caption,subcaption}

\numberwithin{equation}{section}
\theoremstyle{plain}

\newtheorem{theorem}{Theorem}[section]
\newtheorem{lemma}[theorem]{Lemma}

\newtheorem{proposition}[theorem]{Proposition}
\newtheorem{corollary}[theorem]{Corollary}

\newcommand{\argmin}{\mathop{\rm argmin~}}

\newcommand{\qhat} {\widehat{q}}
\newcommand{\qtil} {\widetilde{q}}

\newcommand \bbP{\mathbb{P}}
\newcommand \bbE{\mathbb{E }}

\def\T{{ \mathrm{\scriptscriptstyle T} }}

\def\VB{{ \mathrm{\scriptscriptstyle VB} }}

\def\m{\mathcal}
\def\mb{\mathbb}

\begin{document}
\bibliographystyle{plain}

\title{$\alpha$-Variational Inference with Statistical Guarantees}


\author[1]{Yun Yang \thanks{yyang@stat.fsu.edu}}
\author[2]{Debdeep Pati\thanks{debdeep@stat.tamu.edu}}
\author[2]{Anirban Bhattacharya\thanks{anirbanb@stat.tamu.edu}}
\affil[1]{Department of Statistics, Florida State University}
\affil[2]{Department of Statistics, Texas A\&M University}
\date{\vspace{-2em}}
\maketitle 

\begin{abstract}
We propose a family of variational approximations to Bayesian posterior distributions, called $\alpha$-VB, with provable statistical guarantees. The standard variational approximation is a special case of $\alpha$-VB with $\alpha=1$. When $\alpha \in(0,1]$, a novel class of variational inequalities are developed for linking the Bayes risk under the variational approximation to the objective function in the variational optimization problem, implying that maximizing the evidence lower bound in variational inference has the effect of minimizing the Bayes risk within the variational density family. Operating in a frequentist setup, the variational inequalities imply that point estimates constructed from the $\alpha$-VB procedure converge at an optimal rate to the true parameter in a wide range of problems. 
We illustrate our general theory with a number of examples, including the mean-field variational approximation to (low)-high-dimensional Bayesian linear regression with spike and slab priors, mixture of Gaussian models, latent Dirichlet allocation, and (mixture of) Gaussian variational approximation in regular parametric models.
\end{abstract}


\section{Introduction and preliminaries}

Variational inference \cite{jordan1999introduction,wainwright2008graphical} is a widely-used 
tool for approximating complicated probability densities, especially those arising as posterior distributions from complex hierarchical Bayesian models. It provides an alternative strategy to Markov chain Monte Carlo (MCMC, \cite{hastings1970monte,gelfand1990sampling}) sampling by turning the sampling/inference problem into an optimization problem, where a closest member, relative to the Kullback--Leibler (KL) divergence, in a family of approximate densities is picked out as a proxy to the target density. Variational inference has found its success in a variety of contexts, especially in models involving latent variables, such as Hidden Markov models \cite{mackay1997ensemble}, graphical models \cite{attias2000variational,wainwright2008graphical}, mixture models \cite{humphreys2000approximate,corduneanu2001variational,ueda2002bayesian}, and topic models \cite{blei2012probabilistic,blei2003latent} among others. See the recent review paper \cite{blei2017variational} by Blei et al.~for a comprehensive introduction to variational inference.

The popularity of variational methods can be largely attributed to their computational advantages over MCMC. It has been empirically observed in many applications that variational inference operates orders of magnitude faster than MCMC for achieving the same approximation accuracy. Moreover, compared to MCMC, variational inference tends to be easier to scale to big data due to its inherent optimization nature, and can take advantage of modern optimization techniques such as stochastic optimization \cite{kushner1997stochastic,kingma2014adam} and distributed optimization \cite{ahmed2012scalable}.  However, unlike MCMC that is guaranteed to produce (almost) exact samples from the target density for ergodic chains \cite{robert2004monte}, variational inference does not enjoy such general theoretical guarantee. 

Several threads of research have been devoted to characterize statistical properties of the variational proxy to the true posterior distribution; refer to Section 5.2 of \cite{blei2017variational} for a relatively comprehensive survey of the theoretical literature on variational inference. However, almost all these studies are conducted in a case-by-case manner, by either explicitly analyzing the fixed point equation of the variational optimization problem, or directly analyzing the iterative algorithm for solving the optimization problem. In addition, these analyses require certain structural assumptions on the priors such as conjugacy, and is not applicable to broader classes of priors. 

This article introduces a novel class of variational approximations and studies their large sample convergence properties in a unified framework. The new variational approximation, termed $\alpha$-VB, introduces a fixed temperature parameter $\alpha$ inside the usual VB objective function which controls the relative trade-off between model-fit and prior regularization. The usual VB approximation is retained as a special case corresponding to $\alpha = 1$. The $\alpha$-VB objective function is partly motivated by {\em fractional posteriors} \cite{walker2001bayesian,bhattacharya2016bayesian}; specific connections are drawn in \S2.1. The general $\alpha$-VB procedure also inherits all the computational tractability and scalability from the $\alpha = 1$ case, and implementation-wise only requires simple modifications to existing variational algorithms. 

For $\alpha \in (0, 1]$, we develop novel variational inequalities for the Bayes risk under the variational solution. These variational inequalities link the Bayes risk with the $\alpha$-VB objective function, implying that maximizing the evidence lower bound has the effect of minimizing the Bayes risk within the variational density family. A crucial upshot of this analysis is that point estimates constructed from the variational posterior concentrate at the true parameter at the same rate as those constructed from the actual posterior for a variety of problems. There is now a well-developed literature on the frequentist concentration properties of posterior distributions in nonparametric problems; refer to \cite{rousseau2016frequentist} for a detailed review, and the present paper takes a step towards developing similar general-purpose theoretical guarantees for variational solutions. We applied our theory to a number of examples where VB is commonly used, including mean-field variational approximation to high-dimensional Bayesian linear regression with spike and slab priors, mixtures of Gaussian models, latent Dirichlet allocation, and Gaussian-mixture variational approximation to regular parametric models. 

The $\alpha < 1$ case is of particular interest as the major ingredient of the variational inequality involves the prior mass assigned to appropriate Kullback--Leibler neighborhoods of the truth which can be bounded in a straightforward fashion in the aforesaid models and beyond. We mention here a recent preprint by Alquier and Ridgeway \cite{alquier2017concentration} where variational approximations to tempered posteriors (without latent variables) are conducted. The $\alpha$-VB objective function considered here incorporates a much broader class of models involving latent variables, and the corresponding variational inequality recovers the risk bound of \cite{alquier2017concentration} when no latent variables are present. The variational inequalities for the $\alpha < 1$ case do not immediately extend to the $\alpha = 1$ case under a simple limiting operation, and require a separate treatment under stronger assumptions. In particular, we make use of additional testability assumptions on the likelihood function detailed in \S 3.2. Similar assumptions have been used to study concentration of the usual posterior \cite{ghosal2007convergence}.

It is a well-known fact \cite{wang2005inadequacy,westling2015establishing} that the covariance matrices from the variational approximations are typically ``too small" compared with those for the sampling distribution of the maximum likelihood estimator, which combined with the Bernstein von-Mises theorem \cite{van2000asymptotic} implies that the variational approximation may not converge to the true posterior distribution.
This fact combined with our result illustrate the landscape of variational approximation---minimizing the KL divergence over the variational family forces the variational distribution to concentrate around the truth at the optimal rate (due to the heavy penalty on the tails in the KL divergence); however, the local shape of the obtained density function around the truth can be far away from that of the true posterior due to mis-match between the distributions in the variational family and the true posterior. Overall, our results reveal that concentration of the posterior measure is not only useful in guaranteeing desirable statistical properties, but also has computational benefits in certifying consistency and concentration of variational approximations.

In the remainder of this section, we introduce key notation used in the paper and provide necessary background on variational inference. 
\subsection{Notation}
We briefly introduce notation that will be used throughout the paper.
Let $h(p\,||\,q) =(\int(p^{1/2}- q^{1/2})^2d\mu)^{1/2}$ and $D(p\,||\,q) = \int p\log$ $(p/q)d\mu$ denote the Hellinger distance and Kullback--Leibler divergence, respectively, between two probability density functions $p$ and $q$ relative to 
a common dominating measure $\mu$. We define an additional discrepancy measure $V(p\,||\,q) = \int p\,\log^2(p/q)\,d\mu$, which will be referred to as the $V$-divergence. For a set $A$, we use the notation $I_A$ to denote its indicator function. For any vector $\mu$ and positive semidefinite matrix $\Sigma$, we use $\m N(\mu, \Sigma)$ to denote the normal distribution with mean $\mu$ and covariance matrix $\Sigma$, and use $\m N(\theta;\, \mu, \Sigma)$ to denote its pdf at $\theta$.

For any $\alpha \in (0, 1)$, let
\begin{align}\label{eq:renyi_def}
D_{\alpha}(p \,||\, q) = \frac{1}{\alpha-1} \log \int p^{\alpha} q^{1 - \alpha} d\mu 
\end{align}
denote the R{\'e}nyi divergence of order $\alpha$. Jensen's inequality implies that $D_{\alpha}(p\,||\, q) \ge 0$ for any $\alpha \in (0, 1)$, and the equality holds if and only if $p=q$. The Hellinger distance can be related with the $\alpha$-divergence with $\alpha=1/2$ by $D_{1/2}(p\,||\,  q) = -2 \log \{ 1 - (1/2)h^2(p\,||\,  q) \} \ge h^2(p\,||\,  q)$ using the inequality $\log(1 + t) < t$ for $t > -1$. More details and properties of the $\alpha$-divergence can be found in \cite{van2014renyi}. 

\subsection{Review of variational inference}\label{Sec:VI}
Suppose we have observations $Y^n=(Y_1,\ldots,Y_n)\in \mathscr{Y}^n$ with $n$ denoting the sample size. Let $\mb P_{\theta}^{(n)}$ be the distribution of $Y^n$ given parameter $\theta \in \Theta$ that admits a density $p^{(n)}_\theta$ relative to the Lebesgue measure. We will also interchangeably use $P(Y^n\,|\,\theta)$ and $p(Y^n\,|\,\theta)$ to denote $\mb P_{\theta}^{(n)}$ and its density function (likelihood function) $p_\theta^{(n)}$. Assume additionally that the likelihood $p(Y^n\,|\,\theta)$ can be represented as 
\begin{align*}
p(Y^n \mid \theta) = \sum_{s^n} p(Y^n \mid S^n = s^n, \theta) \, p(S^n = s^n \mid \theta),
\end{align*}
where $S^n$ denotes a collection of latent or unobserved variables; the superscript $n$ signifies the possible dependence of the number of latent variables on $n$; for example, when there are observation specific latent variables. In certain situations, the latent variables may be introduced for purely computational reasons to simplify an otherwise intractable likelihood, such as the latent cluster indicators in a mixture model. Alternatively, a complex probabilistic model $p(Y^n\,|\,\theta)$ may itself be defined in a hierarchical fashion by first specifying the distribution of the data given latent variables and parameters, and then specifying the latent variable distribution given parameters; examples include the latent Dirichlet allocation and many other prominent Bayesian hierarchical models.  For ease of presentation, we have assumed discrete latent variables in the above display and continue to do so subsequently, although our development seamlessly extends to continuous latent variables by replacing sums with integrals; further details are provided in a supplemental document. 

Let $P_{\theta}$ denote a prior distribution on $\theta$ with density function $p_\theta$, and denote $W^n = (\theta, \, S^n) \in \mathscr{W}^n$. In a Bayesian framework, all inference is based on the augmented posterior density $p(W^n\,|\,Y^n)$ given by 
\begin{align}\label{eq:aug_post}
p(W^n\,|\,Y^n) = p(\theta, S^n \,|\, Y^n) \propto p(Y^n \,|\, \theta, S^n) \, p(S^n\,|\,\theta) \, p_{\theta}(\theta). 
\end{align}
In many cases, $p(W^n\,|\, Y^n)$ can be inconvenient for conducting direct analysis due to its intractable normalizing constant and expensive to sample from due to the slow mixing of standard MCMC algorithms. Variational inference aims to bypass these difficulties by turning the inference problem into an optimization problem, which can be solved by using iterative algorithms such as coordinate descent \cite{bishop2006pattern} and alternating minimization.

Let $\Gamma$ denote a pre-specified family of density functions over $\mathscr{W}^n$ that can be either parameterized by some ``variational parameters", or required to satisfy some structural constraints (see below for examples of $\Gamma$).
The goal of variational inference is to approximate this conditional density $p(W^n\,|\, Y^n)$ by finding the closest member of this family in KL divergence to the conditional density $p(W^n\,|\, Y^n)$ of interest, that is, computing the minimizer 
\begin{align}
\qhat_{W^n}:&\, = \argmin_{q_{W^n} \in \Gamma} D\big[\,q_{W^n}(\cdot)\,\big|\big|\, p(\cdot\,|\, Y^n)\,\big] \notag\\
&=\argmin_{q_{W^n} \in \Gamma} \bigg\{-\int_{\mathscr{W}^n} q_{W^n}(w^n)\, \log \frac{p(w^n\,|\, Y^n)}{q_{W^n}(w^n)}\, dw^n \bigg\}\notag \\
&=\argmin_{q_{W^n} \in \Gamma} \bigg\{- \, \underbrace{\int_{\mathscr{W}^n} q_{W^n}(w^n)\, \log \frac{p(Y^n\,|\,w^n)\, p_{W^n}(w^n)}{q_{W^n}(w^n)}\, dw^n}_{L(q_{W^n})} \bigg\} \label{eq:VB_obj_original}
\end{align}
where the last step follows by using Bayes' rule and the fact that the marginal density $p(Y^n)$ does not depend on $W^n$ and $q_{W^n}$. The function $L(q_{W^n})$ inside the argmin-operator above (without the negative sign) is called the evidence lower bound (ELBO, \cite{blei2017variational}) since it provides a lower bound to the log evidence $\log p(Y^n)$,
\begin{align}\label{eq:elbo_dec}
\log p(Y^n) = L(q_{W^n}) + D\big[\,q_{W^n}(\cdot)\,\big|\big|\, p(\cdot\,|\, Y^n)\,\big] \geq  L(q_{W^n}),
\end{align}
where the equality holds if and only if $q_{W^n} = p(\cdot\,|\, Y^n)$. The decomposition \eqref{eq:elbo_dec} provides an alternative interpretation of variational inference to the original derivation from Jensen's inequality\cite{jordan1999introduction}---minimizing
the KL divergence over the variational family $\Gamma$ is equivalent to maximizing the ELBO over $\Gamma$.
When $\Gamma$ is composed of all densities over $\mathscr{W}^n$, this variational approximation $\qhat_{W^n}$ exactly recovers $p(W^n\,|\, Y^n)$. In general, the variational family $\Gamma$ is chosen to balance between the computational tractability and the approximation accuracy. Some common examples of $\Gamma$ are provided below.

\vspace{1em}
\paragraph{{\bf Example:} (Exponential variational family)} 
When there is no latent variable and $W^n=\theta\in\Theta$ corresponds to the parameter in the model, a popular choice of the variational family is an exponential family of distributions. 
Among the exponential variational families, the Gaussian variational family, $
q_\theta(\theta;\, \mu, \,\Sigma) \equiv \m N(\theta;\,\mu,\,\Sigma)$ for $\theta\in \mb R^d$,
is the most widely-used owing to the Bernstein von-Mises theorem (Section 10.2 of \cite{van2000asymptotic}), 
stating that for regular parametric models, the posterior distribution converges to a Gaussian limit relative to the total variation metric as the sample size tends to infinity. 
There are also some recent developments by replacing the single Gaussian with a Gaussian-mixture as the variational family to improve finite-sample approximation \cite{zobay2014variational}, which is useful when the posterior distribution is skewed or far away from Gaussian for the given sample size.

\paragraph{{\bf Example:} (Mean-field variational family)} 
Suppose that $W^n$ can be decomposed into $m$ components (or blocks) as $W^n=(W_1,\,W_2,\ldots,\,W_m)$ for some $m>1$, where each component $W_j \in \mathscr{W}_j$ can be multidimensional. 
The mean-field variational family $\Gamma_{MF}$ is composed of all density functions over $\mathscr{W}^n=\prod_{j=1}^m \mathscr{W}_j$ that factorizes as
\begin{align*}
q_{W^n} (w^n) = \prod_{j=1}^m q_{W_j}(w_j), \quad w^n=(w_1,\ldots,w_m)\in \mathscr{W}^n,
\end{align*}
where each variational factor $q_{W_j}$ is a density function over $\mathscr{W}_j$ for $j=1,\ldots,m$. A natural mean-field decomposition is to let $q_{W^n}(w^n) = q_{\theta}(\theta) \, q_{S^n}(s^n)$, with $q_{S^n}$ often further decomposed as $q_{S^n}(s^n) = \prod_{i=1}^n q_{S_i}(s_i)$. 

Note that we have not specified the parametric form of the individual variational factors, which are determined by properties of the model---
in some cases, the optimal $q_{W_j}$ is in the same parametric family as the conditional distribution of $W_j$ given the parameter.
The corresponding mean-field variational approximation $\qhat_{W^n}$, which is necessarily of the form $\prod_{j=1}^m \qhat_{W_j}(w_j)$, can be computed via the coordinate ascent variational inference (CAVI) algorithm \cite{bishop2006pattern,blei2017variational} which iteratively optimizes each variational factor keeping others fixed at their present value and resembles the EM algorithm in the presence of latent variables.

The mean-field variational family can be further constrained by restricting each factor $q_{W_j}$ to belong to a parametric family, such as the exponential family in the previous example. In particular, it is a common practice to restrict the variational density $q_\theta$ of the parameter into a structured family (for example, the mean-field family if $\theta$ is multi-dimensional), which will be denoted by $\Gamma_\theta$ in the sequel.  \\[2ex]
The rest of the paper is organized as follows. In \S2, we introduce the $\alpha$-VB objective function and relate it to usual VB. \S3 presents our general theoretical results concerning finite sample risk bounds for the $\alpha$-VB solution. In \S4, we apply the theory to concrete examples. We conclude with a discussion in \S5. All proofs and some additional discussions are provided in a separate supplemental document. The supplemental document also contains a detailed simulation study.


\section{The $\alpha$-VB procedure}
Before introducing the proposed family of objective functions, we first represent the KL term $D\big[\,q_{W^n}(\cdot)\,\big|\big|\, p(\cdot\,|\, Y^n)\,\big]$ in a more convenient form which provides intuition into how VB works in the presence of latent variables and aids our subsequent theoretical development. 
\subsection{A further decomposition of the ELBO}

To aid our subsequent development, we introduce some additional notation and make some simplifying assumptions. First, we decompose $\theta=(\mu,\,\pi)$, with $p(Y^n \,|\, S^n = s^n, \theta) = p(Y^n\,|\,S^n = s^n, \mu)$ and 
and $\pi_{s^n} :\,= p(S^n = s^n \,|\, \theta)$. In other words, $\mu$ is the parameter characterizing the conditional distribution $P(Y^n\,|\,S^n,\,\mu)$ of the observation $Y^n$ given latent variable $S^n$, and $\pi = (\pi_{s^n})$ characterizes the distribution $P(S^n\,|\,\pi)$ of the latent variables. We shall also assume the mean-field decomposition 
\begin{align}\label{eq:MFa}
q_{W^n}(w^n) = q_{\theta}(\theta) \, q_{S^n}(s^n)
\end{align} throughout, and let $\Gamma = \Gamma_{\theta} \times \Gamma_{S^n}$ denote the class of such product variational distributions. When necessary subsequently, we shall further assume $q_{S^n}(s^n) = \prod_{i=1}^n q_{S_i}(s_i)$ and $q_{\theta}(\theta) = q_{\mu}(\mu) \, q_{\pi}(\pi)$, which however is not immediately necessary for this subsection. 

The KL divergence $D\big[\,q_{W^n}(\cdot)\,\big|\big|\, p(\cdot\,|\, Y^n)\,\big]$ in \eqref{eq:VB_obj_original} involves both parameters and latent variables. Separating out the KL divergence for the parameter part leads to the equivalent representation 
\begin{align}\label{eq:KL_decomp}
D\big[\,q_{W^n}(\cdot)\,\big|\big|\, p(\cdot\,|\, Y^n)\,\big] = D(q_{\theta}\,\big|\big|\,p_{\theta}) - \int_{\Theta} \underbrace{ \bigg[ \sum_{s^n} q_{S^n}(s^n) \, \log \frac{p(Y^n\,|\,\mu, s^n) \, \pi_{s^n}}{q_{S^n}(s^n)} \bigg] }_{\widehat{\ell}_n(\theta)} \, q_{\theta}(d\theta).
\end{align}
Observe that, using concavity of $x \mapsto \log x$ and Jensen's inequality, 
\begin{align*}
\log p(Y^n\,|\,\theta) = \log \bigg[\sum_{s^n} q_{S^n}(s^n) \, \frac{p(Y^n\,|\,\mu, s^n) \, \pi_{s^n}}{q_{S^n}(s^n)} \bigg] \ge \sum_{s^n} q_{S^n}(s^n) \, \log \frac{p(Y^n\,|\,\mu, s^n) \, \pi_{s^n}}{q_{S^n}(s^n)}. 
\end{align*}
The quantity $\widehat{\ell}_n(\theta)$ in \eqref{eq:KL_decomp} can therefore be recognized as an approximation (from below) to the log likelihood $\ell_n(\theta) :\,= \log p(Y^n \mid \theta)$ in terms of the latent variables. Define an average Jensen gap $\Delta_J$ due to the variational approximation to the log-likelihood, 
\begin{align*}
\Delta_J(q_{\theta}, q_{S^n}) =  \int_{\Theta}  \big[ \ell_n(\theta) - \widehat{\ell}_n(\theta) \big]  \, q_{\theta}(d\theta) \ge 0.
\end{align*}
With this, write the KL divergence $D\big[\,q_{W^n}(\cdot)\,\big|\big|\, p(\cdot\,|\, Y^n)\,\big]$ as 
\begin{align}\label{eq:KL_decomp1}
& D\big[\,q_{W^n}(\cdot)\,\big|\big|\, p(\cdot\,|\, Y^n)\,\big] = - \int_{\Theta} \ell_n(\theta) q_{\theta}(d\theta)  + \Delta_J(q_{\theta}, q_{S^n}) + D(q_{\theta}\,\big|\big|\,p_{\theta}), 
\end{align}
which splits as a sum of three terms: an integrated (w.r.t.~the variational distribution) negative log-likelihood, the KL divergence between the variational distribution $q_{\theta}$ and the prior $p_{\theta}$ for $\theta$, and the Jensen gap $\Delta_J$ due to the latent variables. In particular, the role of the latent variable variational distribution $q_{S^n}$ is conveniently confined to $\Delta_J$. 

Another view of the above is an equivalent formulation of the ELBO decomposition \eqref{eq:elbo_dec}, 
\begin{align}\label{eq:ELBO_mod}
\log p(Y^n) = L(q_{W^n}) + \Delta_J(q_{\theta}, q_{S^n}) + D\big[q_{\theta}(\theta)\,\big|\big|\, p(\theta\,|\,Y^n)\big]. 
\end{align}
which readily follows since 
\begin{align*}
D\big[q_{\theta}(\theta)\,\big|\big|\, p(\theta\,|\,Y^n)\big] = - \int_{\Theta} \ell_n(\theta) q_{\theta}(d\theta)  + D(q_{\theta}\,\big|\big|\,p_{\theta}). 
\end{align*}
Thus, in latent variable models, maximizing the ELBO $L(q_{W^n})$ is equivalent to minimizing a sum of the Jensen gap $\Delta_J$ and the KL divergence between the variational density and the posterior density of the parameters. When there is no likelihood approximation with latent variables, $\Delta_J = 0$. 

\subsection{The $\alpha$-VB objective function}

Here and in the rest of the paper, we adopt the frequentist perspective by assuming that there is a true data generating model $\mb P_{\theta^\ast}^{(n)}$ that generates the data $Y^n$, and $\theta^\ast$ will be referred to as the true parameter, or simply truth. Let $\ell_n(\theta, \theta^\ast) = \ell_n(\theta) - \ell_n(\theta^\ast)$ be the log-likelihood ratio. Define 
\begin{align}\label{eq:VBapp}
\Psi_n(q_{\theta}, \,q_{S^n}) = - \int_{\Theta} \ell_n(\theta, \theta^\ast) q_{\theta}(d\theta)  + \Delta_J(q_{\theta}, q_{S^n}) + D(q_{\theta}\,\big|\big|\,p_{\theta}),
\end{align}
and observe that $\Psi_n$ differs from the KL divergence $D\big[\,q_{W^n}(\cdot)\,\big|\big|\, p(\cdot\,|\, Y^n)\,\big]$ in \eqref{eq:KL_decomp1} only by $\ell_n(\theta^\ast)$ which does not involve the variational densities. Hence, minimizing $D\big[\,q_{W^n}(\cdot)\,\big|\big|\, p(\cdot\,|\, Y^n)\,\big]$ is equivalent to minimizing $\Psi_n(q_{\theta}, \,q_{S^n})$. We note here that the introduction of the $\ell_n(\theta^\ast)$ term is to develop theoretical intuition and the actual minimization does not require the knowledge of $\theta^\ast$.

The objective function $\Psi_n$ in \eqref{eq:VBapp} elucidates the trade-off between model-fit and fidelity to the prior underlying a variational approximation, which is akin to the classical bias-variance trade-off for shrinkage or penalized estimators. The model-fit term consists of two constituents: the first term is an averaged (with respect to the variational distribution) log-likelihood ratio which tends to get small as the variational distribution $q_{\theta}$ places more mass near the true parameter $\theta^\ast$, while the second term is the Jensen gap $\Delta_J$ due to the variational approximation with the latent variables. On the other hand, the regularization or penalty term $D(q_\theta\,| |\, p_\theta)$ prevents over-fitting to the data by constricting the KL divergence between the variational solution and the prior.

In this article, we study a wider class of variational objective functions $\Psi_{n,\alpha}$ indexed by a scalar parameter $\alpha \in (0, 1]$ which encompass the usual VB, 
\begin{align}\label{eq:VBapp1}
\Psi_{n,\alpha}(q_{\theta}, \, q_{S^n}) = \underbrace{- \int_{\Theta} \ell_n(\theta, \theta^\ast) q_{\theta}(d\theta)  + \Delta_J(q_{\theta}, q_{S^n})}_{\text{model fit}} + \underbrace{\alpha^{-1} D(q_{\theta}\,\big|\big|\,p_{\theta})}_{\text{regularization}},
\end{align}
and define the $\alpha$-VB solution as 
\begin{align}\label{eq:alpha_VB_sol}
(\qhat_{\theta,\alpha}, \qhat_{S^n, \alpha}) = \argmin_{(q_{\theta}, q_{S^n}) \in \Gamma} \Psi_{n,\alpha}(q_{\theta},\, q_{S^n}). 
\end{align}
Observe that the $\alpha$-VB criterion $\Psi_{n,\alpha}$ differs from $\Psi_n$ only in the regularization term, where the inverse temperature parameter $\alpha$ controls the amount of regularization, with smaller $\alpha$ implying a stronger penalty. When $\alpha = 1$, $\Psi_{n,\alpha}$ reduces to the usual variational objective function $\Psi_n$ in \eqref{eq:VBapp}, and we shall denote the solution of \eqref{eq:alpha_VB_sol} by $\qhat_{\theta}$ and $\qhat_{S^n}$ as before. As we shall see in the sequel, the introduction of the temperature parameter $\alpha$ substantially simplifies the theoretical analysis and allows one to certify (near-)minimax optimality of the $\alpha$-VB solution for $\alpha < 1$ under only a prior mass condition, whereas analysis of the the usual VB solution ($\alpha = 1$) requires more intricate testing arguments. 

The $\alpha$-VB solution can also be interpreted as the minimizer of a certain divergence function between the product variational distribution $q_{\theta}(\theta) \times q_{S^n}(s^n)$ and the joint $\alpha$-{\em fractional posterior} distribution \cite{bhattacharya2016bayesian} of $(\theta, S^n)$, 
\begin{align}\label{eq:jfp}
P_\alpha(\theta\in B, s^n\,|\,Y^n) = \frac{\int_{B} \big[p(Y^n\,|\,\mu,\,s^n)\,\pi_{s^n}\big]^\alpha\, p_\theta(\theta)\, d\theta}{\int_{\Theta} \sum_{s^n} \big[p(Y^n\,|\,\mu,\,s^n)\,\pi_{s^n}\big]^\alpha\, p_\theta(\theta)\, d\theta},
\end{align}
which is obtained by raising the joint likelihood of $(\theta, s^n)$ to the fractional power $\alpha$, and combining with the prior $p_\theta$ using Bayes' rule. We shall use $p_{\alpha}(\cdot \,|\, Y^n)$ to denote the fractional posterior density. 
The fractional posterior is a specific example of a Gibbs posterior \cite{jiang2008gibbs} and shares a nice coherence property with the usual posterior when viewed as a mechanism for updating beliefs \cite{bissiri2016general}.

\begin{proposition}[Connection with fractional posteriors]
The $\alpha$-VB solution $(\qhat_{\theta,\alpha}, \qhat_{S^n, \alpha})$ satisfy, 
\begin{align*}
(\qhat_{\theta,\alpha}, \qhat_{S^n, \alpha}) = \argmin_{(q_{\theta}, q_{S^n}) \in \Gamma} \bigg[ D\big[\,q_{W^n}(\cdot)\,\big|\big|\, p_{\alpha}(\cdot\,|\, Y^n)\,\big] + (1 - \alpha) \m H(q_{S^n}) \bigg],
\end{align*}
where $\m H(q_{S^n}) = - \sum_{s^n} q_{S^n}(s^n) \log q_{S^n}(s^n)$ is the {\em entropy} of $q_{S^n}$, and $p_{\alpha}(\cdot \,|\, Y^n)$ is the joint $\alpha$-fractional posterior density of $w^n = (\theta, s^n)$.
\end{proposition}
The entropy term $\m H(q_{S^n})$ encourages the latent-variable variational density $q_{S^n}$ to be concentrated to the uniform distribution, in addition to minimizing the KL divergence between $q_{W^n}(\cdot)$ and $p_{\alpha}(\cdot\,|\, Y^n)$. In particular, if there are no latent variables, the entropy term disappears and the objective function reduces to a KL divergence between $q_{\theta}$ and $p_{\alpha}(\theta\,|\,Y^n)$. 

We conclude this section by remarking that the additive decomposition of the model-fit term in \eqref{eq:VBapp1}  provides a peak into why mean-field approximations work for latent variable models, since the roles of the variational density $q_{S^n}$ for the latent variables and $q_{\theta}$ for the model parameters are de-coupled. Roughly speaking, a good choice of $q_{S^n}$ should aim to make the Jensen gap $\Delta_J$ small, while the choice of $q_{\theta}$ should balance the integrated log-likelihood ratio and the penalty term. This point is crucial for the theoretical analysis. 

\section{Variational risk bounds for $\alpha$-VB}

In this section, we investigate concentration properties of the $\alpha$-VB posterior under a frequentist framework assuming the existence of a true data generating parameter $\theta^\ast$. We first focus on the $\alpha < 1$ case, and then separately consider the $\alpha = 1$ case.  
The main take-away message from our theoretical results below is that under fairly general conditions, the $\alpha$-VB procedure concentrates at the true parameter at the same rate as the actual posterior, and as a result, point estimates obtained from the $\alpha$-VB can provide rate-optimal frequentist estimators. These results thus compliment the empirical success of VB in a wide variety of models. 

We present our results in the form of Bayes risk bounds for the variational distribution. Specifically, for a suitable loss function $r(\theta, \, \theta^\ast)$, we aim to obtain a high-probability (under the data generating distribution $\mb P^{(n)}_{\theta^\ast}$) to the variational risk 
\begin{align}\label{eq:VB_risk}
\int r(\theta, \, \theta^\ast) \, \qhat_{\theta, \alpha}(d\theta). 
\end{align}
In particular, if $r(\cdot,\, \cdot)$ is convex in its first argument, then the above risk bound immediately translates into a risk bound for the $\alpha$-VB point estimate $\widehat{\theta}_{\VB, \alpha} = \int \theta \, \qhat_{\theta, \alpha}(d\theta)$ using Jensen's inequality:
$$
r(\widehat{\theta}_{\VB, \alpha}, \, \theta^\ast) \le \int r(\theta,\, \theta^\ast) \, \qhat_{\theta, \alpha}(d\theta). 
$$
Specifically, our goal will be to establish general conditions under which $\widehat{\theta}_{\VB, \alpha}$ concentrates around $\theta^\ast$ at the minimax rate for the particular problem. 

\subsection{Risk bounds for the $\alpha < 1$ case:} We use the shorthand 
\begin{align*}
\frac{1}{n}\,D^{(n)}_{\alpha}(\theta,\,\theta^\ast):\,=\frac{1}{n}\, D_\alpha\big[p^{(n)}_\theta\,\big|\big|\,p^{(n)}_{\theta^\ast}\big]
\end{align*}
to denote the averaged $\alpha$-divergence between $\mb P^{(n)}_\theta$ and $\mb P^{(n)}_{\theta^\ast}$. We adopt the theoretical framework of \cite{bhattacharya2016bayesian} to use this divergence as our loss function $r(\theta, \, \theta^\ast)$ for measuring the closeness between any $\theta\in\Theta$ and the truth $\theta^\ast$. Note that in case of i.i.d.~observations, this averaged divergence $n^{-1}\,D^{(n)}_{\alpha}(\theta,\,\theta^\ast)$ simplifies to $D_\alpha\big[p_{\theta}\,||\,p_{\theta^\ast}\big]$, which is stronger than the squared Hellinger distance $h^2\big[p_{\theta}\,||\,p_{\theta^\ast}\big]$ between $p_{\theta}$ and $p_{\theta^\ast}$ for any fixed $\alpha\in[1/2,1)$.

Our first main result provides a general finite-sample upper bound to the variational Bayes risk \eqref{eq:VB_risk} for the above choice of $r(\theta,\, \theta^\ast)$. 
\begin{theorem}[Variational risk bound]\label{thm:main}
Recall the $\alpha$-VB objective function $\Psi_{n,\alpha}(q_{\theta}, \, q_{S^n})$ from \eqref{eq:VBapp1}. For any $\zeta \in (0, 1)$, it holds with $\mb P_{\theta^\ast}^n$ probability at least $(1-\zeta)$ that for any probability measure $q_\theta \in\Gamma_{\theta}$ with $q_{\theta} \ll p_{\theta}$ and any probability measure $q_{S^n} \in \Gamma_{S^n}$ on $S^n$, 
\begin{align*}
\int \frac{1}{n}\,D^{(n)}_{\alpha}(\theta,\,\theta^\ast) \ \qhat_{\theta,\alpha}(\theta)\,d\theta
\le \frac{\alpha}{n(1-\alpha)} \Psi_{n,\alpha}(q_{\theta},\, q_{S^n})  + \frac{1}{n(1-\alpha)}  \log(1/\zeta).
\end{align*}
\end{theorem}
Here and elsewhere, the probability statement is uniform over all $(q_{\theta}, q_{S^n}) \in \Gamma$. Theorem~\ref{thm:main} links the variational Bayes risk for the $\alpha$-divergence to the objective function $\Psi_{n,\alpha}$ in \eqref{eq:VBapp1}. As a consequence, minimizing $\Psi_{n,\alpha}$ in \eqref{eq:VBapp1} has the same effect as minimizing the variational Bayes risk. To apply Theorem \ref{thm:main} to various problems, we now discuss strategies to further analyze and simplify $\Psi_{n,\alpha}$ under appropriate structural constraints of $\Gamma_\theta$ and $\Gamma_{S^n}$. To that end, we make some simplifying assumptions.

First, we assume a further mean-field decomposition $q_{S^n}(s^n) =\prod_{i=1}^n q_{S_i}(s_i)$ for the latent variables $S^n$, where each factor $q_{S_i}$ is restriction-free. Second, the inconsistency of the mean-field approximation for state-space models proved in \cite{wang2004lack} indicates that this mean-field approximation for the latent variables may not generally work for non-independent observations with non-independent latent variables. For this reason, we assume that the observation latent variable pair $(S_i,\,Y_i)$ are mutually independent across $i=1,2,\ldots,n$. In fact, we assume that $(S_i,\,Y_i)$ are i.i.d.~copies of $(S,\,Y)$ whose density function is given by $p(S,\,Y\,|\,\mu,\,\pi)=p(Y\,|\,S,\,\mu)\,p(S\,|\,\pi)$. 
Following earlier notation, let $\pi_S :\,= p(S\,|\,\pi)$ denote the probability mass function of the i.i.d.~discrete latent variables $\{S_i\}$, with the parameter $\pi = (\pi_1,\,\pi_2,\ldots,\,\pi_K)$ residing in the $K$-dim simplex $\m S_K=\{\pi\in [0,1]^K:\, \sum_{k}\pi_k=1\}$. Finally, we assume the variational family $\Gamma_\theta$ of the parameter decomposes into $\Gamma_{\mu}\otimes \m S_K$, where $\Gamma_\mu$ denotes variational family for parameter $\mu$.

Let $p(Y\,|\,\theta) = \sum_{s=1}^K \pi_{s}\,p(Y\,|\,\theta, S = s)$ denote the marginal probability density function of the i.i.d.~observations $\{Y_i\}$. The i.i.d.~assumption implies a simplified structure of various quantities encountered before, e.g.~$\pi_{S^n}=\prod_{i=1}^n\pi_{s_i}$, $p(Y^n\,|\,\mu,S^n) = \prod_{i=1}^n \pi_{s_i}p(Y_i\,|\,\mu,S_i)$ and $p(Y^n\,|\,\theta)=\prod_{i=1}^n p(Y_i\,|\,\theta)$. Moreover, under these assumptions, $n^{-1} \, D^{(n)}_{\alpha}(\theta,\,\theta^\ast) = D_{\alpha}\big[p(\cdot\,|\,\theta)\,\big|\big|\, p(\cdot\,|\, \theta^\ast)\big]$.

As discussed in the previous subsection, the decoupling of the roles of $q_{\theta}$ and $q_{S^n}$ in the model fit term aid bounding $\Psi_{n,\alpha}$. Specifically, we first choose a $\widetilde{q}_{S^n}$ which controls the Jensen gap $\Delta_J$, and then make a choice of $q_{\theta}$ which controls $\Psi_{n,\alpha}(q_{\theta}, \, \widetilde{q}_{S^n})$. The choice of $q_{\theta}$ requires a delicate balance between placing enough mass near $\theta^\ast$ and controlling the KL divergence from the prior. 

For a fixed $q_{\theta}$, if we choose $q_{S^n}$ to be the full conditional distribution of $S^n$ given $\theta$, i.e., 
\begin{align*}
q_{S^n}(s^n \mid \theta) = \prod_{i=1}^n q_{S_i}(s_i \mid \theta)=\prod_{i=1}^n \frac{\pi_{s_i}\, p(Y_i\,|\,\mu, \,s_i)}{p(Y_i\,|\,\theta)}, \quad s^n\in \{1,2,\ldots,K\}^n, 
\end{align*}
then the normalizing constant of $q_{S_i}(\cdot \mid \theta)$ is $\sum_{s_i} \pi_{s_i}\, p(Y_i\,|\,\mu, \,S_i) =  p(Y_i\,|\, \theta)$, and as a result, the Jensen gap $\Delta_J = 0$. The mean-field approximation precludes us from choosing $q_{S^n}$ dependent on $\theta$, and hence the Jensen gap cannot be made exactly zero in general. However, this naturally suggests replacing $\theta$ by $\theta^\ast$ in the above display and choosing $\widetilde{q}_{S_i} \propto \pi^\ast_{s_i}\, p(Y_i\,|\,\mu^\ast, \,S_i)$. This leads us to the following corollary. 

\begin{corollary}[i.i.d.~observations]\label{coro:main}
It holds with $\mb P_{\theta^\ast}^n$ probability at least $(1-\zeta)$ that for any probability measure $q_\theta \in\Gamma_{\theta}$ with $q_{\theta} \ll p_{\theta}$
\begin{equation}\label{Eqn:key_var}
\begin{aligned}
& \int \Big\{D_{\alpha}\big[p(\cdot\,|\,\theta)\,\big|\big|\, p(\cdot\,|\, \theta^\ast)\big] \Big\}\,\qhat_{\theta,\alpha}(\theta)\,d\theta
\le \frac{\alpha}{n(1 - \alpha)} \Psi_{n,\alpha}(q_{\theta},\,\qtil_{S^n}) + \frac{1}{n(1-\alpha)}  \log(1/\zeta),\\
 & = \frac{\alpha}{n(1 - \alpha)} \bigg[ - \int_\Theta  \sum_{i=1}^n \sum_{s_i}\qtil_{S_i}(s_i)\,  \log\frac{p(Y_i\,|\,\mu,s_i) \, \pi_{s_i}}{p(Y_i\,|\,\mu^\ast,s_i) \, \pi^\ast_{s_i}}\, q_\theta(d\theta) + \frac{D(q_\theta\, ||\, p_\theta)}{\alpha}  + \frac{\log(1/\zeta)}{\alpha}\bigg], 
\end{aligned}
\end{equation}
where $\qtil_{S^n}$ is the probability distribution over $S^n$ defined as 
\begin{align}\label{Eqn:q_S}
\qtil_{S^n}(s^n) = \prod_{i=1}^n \qtil_{S_i}(s_i)=\prod_{i=1}^n \frac{\pi^\ast_{s_i}\, p(Y_i\,|\,\mu^\ast, \,s_i)}{p(Y_i\,|\,\theta^\ast)}, \quad s^n\in \{1,2,\ldots,K\}^n.
\end{align}
\end{corollary}
The second line of \eqref{Eqn:key_var} follows from the first since 
\begin{align*}
\Delta_J(q_{\theta}, \qtil_{S^n}) = - \int_\Theta  \sum_{i=1}^n \sum_{s_i}\qtil_{S_i}(s_i)\,  \log\frac{p(Y_i\,|\,\mu,s_i) \, \pi_{s_i}}{p(Y_i\,|\,\mu^\ast,s_i) \, \pi^\ast_{s_i}}\, q_{\theta}(d\theta) + \int \ell_n(\theta, \theta^\ast) \,q_{\theta}(d\theta). 
\end{align*}
After choosing $\qtil_{S^n}$ as \eqref{Eqn:q_S} in Corollary~\ref{coro:main}, we can make the first term in the r.h.s.~of \eqref{Eqn:key_var} small by choosing the variational factor $q_\theta$ of $\theta$ concentrated around $\theta^\ast$. In the rest of this subsection, we will apply Corollary~\ref{coro:main} to derive more concrete variational Bayes risk bounds under some further simplifying assumptions.

As a first application, assume there is no latent variable in the model, that is, $W^n=\theta=\mu$. As discussed before, the $\alpha$-VB solution in this case coincides with the nearest KL point to the $\alpha$-fractional posterior of the parameter. A reviewer pointed out a recent preprint by Alquier and Ridgeway \cite{alquier2017concentration} where they exploit risk bounds for fractional posteriors developed in \cite{bhattacharya2016bayesian} to analyze tempered posteriors and their variational approximations, which coincides with the $\alpha$-VB solution when $W^n = \theta$. The following Theorem \ref{thm:NoLatent} arrives at a similar conclusion to Corollary 2.3 of \cite{alquier2017concentration}. We reiterate here that our main motivation is models with latent variables not considered in \cite{alquier2017concentration}, and Theorem \ref{thm:NoLatent} follows as a corollary of our general result in Theorem \ref{thm:main}. 

\begin{theorem}[No latent variable]\label{thm:NoLatent}
It holds with $\mb P_{\theta^\ast}^n$ probability at least $(1-\zeta)$ that for any probability measure $q_\theta \in\Gamma_{\theta}$ with $q_{\theta} \ll p_{\theta}$
\begin{equation}\label{Eqn:key_var_nolat}
\begin{aligned}
& \int \Big\{D_{\alpha}\big[p(\cdot\,|\,\theta)\,\big|\big|\, p(\cdot\,|\, \theta^\ast)\big] \Big\}\,\qhat_{\theta,\alpha}(\theta)\,d\theta
\\
 & = \frac{\alpha}{n(1 - \alpha)} \bigg[ - \int_\Theta \log\frac{p(Y^n\,|\,\theta)}{p(Y^n\,|\,\theta^\ast)}\, q_\theta(\theta)\,d\theta + \frac{D(q_\theta\, ||\, p_\theta)}{\alpha}  + \frac{\log(1/\zeta)}{\alpha}\bigg], 
\end{aligned}
\end{equation}
\end{theorem}
We will illustrate some particular choices of $q_\theta$ for typical variational families $\Gamma_\Theta$ in the examples in Section 4. 

As a second application, we consider a special case when $\Gamma_\theta$ is restriction-free, which is an ideal example for conveying the general idea of how to choose $q_\theta$ to control the upper bound in~\eqref{Eqn:key_var}. To that end, define two KL neighborhoods around $(\pi^\ast,\,\mu^\ast)$ with radius $(\varepsilon_\pi,\,\varepsilon_\mu)$ as
\begin{equation}\label{Eqn:KLN}
\begin{aligned}
\m B_n(\pi^\ast,\,\varepsilon_\pi) &= \Big\{D(\pi^\ast\, ||\, \pi) \leq \varepsilon_\pi^2,\ \ V(\pi^\ast\, ||\, \pi) \leq \varepsilon_\pi^2\Big\}, \\
\m B_n(\mu^\ast,\,\varepsilon_\mu) &=  \Big\{\sup_{s}D\big[ p(\cdot\,|\,\mu^\ast,s)\, \big|\big| \, p(\cdot\,|\,\mu,s)\big]\leq \varepsilon_\mu^2,\ \ \sup_s V\big[ p(\cdot\,|\,\mu^\ast,s)\, \big|\big| \, p(\cdot\,|\,\mu,s)\big] \leq \varepsilon_\mu^2\Big\},
\end{aligned}
\end{equation}
where we used the shorthand $D(\pi^\ast\, ||\, \pi)=\sum_s \pi^\ast_s\,\log(\pi^\ast_s/\pi_s)$ to denote the KL divergence between categorical distributions with parameters $\pi^\ast\in\m S_K$ and $\pi\in\m S_K$ in the $K$-dim simplex $\m S_K$. By choosing $q_\theta$ as the restriction of $p_\theta$ into $\m B_n(\pi^\ast,\,\varepsilon_\pi)\times \m B_n(\mu^\ast,\,\varepsilon_\mu)$, we obtain the following theorem. Here, we make the assumption of independent priors on $\mu$ and $\pi$, i.e., $p_\theta=p_\mu \otimes p_{\pi}$, to simplify the presentation.

\begin{theorem}[Parameter restriction-free]\label{thm:OI_mixture}
For any fixed $(\varepsilon_\pi,\,\varepsilon_\mu) \in (0, 1)^2$ and $D>1$, with $\bbP_{\theta^\ast}^{(n)}$ probability at least $1 - 5/\{(D-1)^2 \,n \, (\varepsilon_\pi^2+\varepsilon_\mu^2)\}$, it holds that
\begin{equation}\label{eq:OI_mixture}
\begin{aligned}
 & \int \Big\{D_{\alpha}\big[p(\cdot\,|\,\theta)\,\big|\big|\,p(\cdot\,|\, \theta^\ast)\big] \Big\}\,\qhat_{\theta,\alpha}(d\theta) \le \, \frac{D\, \alpha}{1-\alpha} \, (\varepsilon_\pi^2+\varepsilon_\mu^2) \, + \\ 
 &\quad \Big\{ - \frac{1}{n(1-\alpha)} \log P_\pi\big[\m B_n(\pi^\ast,\,\varepsilon_\pi)\big] \Big\} 
 + \Big\{ - \frac{1}{n(1-\alpha)} \log P_\mu\big[B_n(\mu^\ast,\,\varepsilon_\mu) \big] \Big\}.
\end{aligned}
\end{equation}
\end{theorem}
Although the results in this section assume discrete latent variables, similar results can be seamlessly obtained for continuous latent variables; see the supplemental document for more details. We will apply this theorem for analysing mean-field approximations for the Gaussian mixture model and the latent Dirichlet allocation in Section 4.

Observe that the variational risk bound in Theorem \ref{thm:OI_mixture} depends only on prior mass assigned to appropriate KL neighborhoods of the truth. This renders an application of Theorem \ref{thm:OI_mixture} to various practical problems particularly straightforward. As we shall see in the next subsection, the $\alpha = 1$ case, i.e.~the regular VB, requires more stringent conditions involving the existence of exponentially consistent tests to separate points in the parameter space. The testing condition is even necessary for the actual posterior to contract; see, e.g., \cite{bhattacharya2016bayesian}, and hence one cannot avoid the testing assumption for its usual variational approximation. Nevertheless, we show below that once the existence of such tests can be verified, the regular VB approximation can also be shown to contract optimally. 

\subsection{Risk bounds for the $\alpha = 1$ case} We consider any loss function $r(\theta,\,\theta^\ast)$ satisfying the following assumption.

\paragraph{Assumption T (Statistical identifiability):} For some $\varepsilon_n >0$ and any $\varepsilon\geq \varepsilon_n$, there exists a sieve set $\m F_{n,\varepsilon} \subset \Theta$ and a test function $\phi_{n,\varepsilon}:\, \m Y^n \to [0,1]$ such that
\begin{align}
&P_\theta(\m F_{n,\varepsilon}^c) \leq e^{-c\,n\,\varepsilon^2},\label{Eqn:complimentB}\\
&\mb E_{\theta^\ast}[\phi_{n,\varepsilon} ] \leq e^{-c\,n\,\varepsilon_n^2},\label{Eqn:T1}\\
& \mb E_{\theta}[1- \phi_{n,\varepsilon}]  \leq e^{-c\,n\,r(\theta,\,\theta^\ast)},\quad \forall \, \theta\in\m F_{n,\varepsilon} \mbox{ satisfies } r(\theta,\theta^\ast) \geq \varepsilon^2. \label{Eqn:T2}
\end{align}

Roughly speaking, the sieve set $\m F_{n,\varepsilon}$ can be viewed as the effective support of the prior distribution at sample size $n$, and $\varepsilon_n$ the contraction rate of the usual posterior distribution. The first condition~\eqref{Eqn:complimentB} allows us to focus attention to this important region in the parameter space that is not too large, but still possesses most of the prior mass.  The last two conditions~\eqref{Eqn:T1} and \eqref{Eqn:T2} ensure the statistical identifiability of the parameter under the loss $r(\cdot,\,\cdot)$ through the existence of a test function $\phi_{n,\varepsilon}$, and require a sufficiently fast decay of the Type I/II error. In the case when $\Theta$ is compact and $r(\theta,\,\theta')=h^2(\theta\,||\,\theta^\ast)$ is the squared Hellinger distance between $p_{\theta}$ and $p_{\theta^\ast}$, such a test $\phi_{n,\varepsilon}$ always exists \cite{ghosal2007convergence}. A similar set of assumptions are used for showing the concentration of the usual posterior (for example, see \cite{ghosal2000} and \cite{ghosal2007convergence}), with the existence of such sieve sets and test functions verified for numerous model-prior combinations. The only difference in our case is that Assumption T requires the existence of the pair $(\m F_{n,\varepsilon}, \,\phi_{n,\varepsilon})$ for all $\varepsilon\geq \varepsilon_n$, not just at $\varepsilon=\varepsilon_n$. However, this extra requirement appears mild since in most cases a construction of $(\m F_{n,\varepsilon}, \,\phi_{n,\varepsilon})$ at $\varepsilon=\varepsilon_n$ naturally extends to any $\varepsilon\geq \varepsilon_n$. 

Our main result for the usual VB ($\alpha = 1$) provides a finite-sample upper bound to the variational Bayes risk for any loss function $r(\theta,\,\theta^\ast)$ satisfying Assumption T. Here, we use $Q_\theta$ to denote the probability distribution associated with any member $q_\theta$ in the variational density family $\Gamma$.

\begin{theorem}\label{Thm:RegularPosterior}
Under Assumption T, for any $\varepsilon\geq \varepsilon_n$, we have that with $\bbP_{\theta^\ast}^{(n)}$ probability at least $1 - 2e^{-c\,n\,\varepsilon_n^2/2}$, it holds that for any probability measure $q_\theta \in\Gamma_{\theta}$ with $q_{\theta} \ll p_{\theta}$ and any probability measure $q_{S^n} \in \Gamma_{S^n}$ on $S^n$ that 
\begin{equation}\label{Eqn:RegVOI}
\begin{aligned}
&\frac{1}{n}\,\bigg[\widehat{Q}_\theta(\m F_{n,\varepsilon}^c)\,\log \frac{\widehat{Q}_\theta(\m F_{n,\varepsilon}^c)}{P_\theta(\m F_{n,\varepsilon}^c)} +(1-\widehat{Q}_\theta(\m F_{n,\varepsilon}^c))\,\log \frac{1-\widehat{Q}_\theta(\m F_{n,\varepsilon}^c)}{1-P_\theta(\m F_{n,\varepsilon}^c)}\bigg] \\
&\qquad\qquad+c\,\int_{\theta\in\m F_{n,\varepsilon},\, r(\theta,\,\theta^\ast) \geq \varepsilon^2} r(\theta,\,\theta^\ast)\, \qhat_\theta(\theta)\,d\theta \leq \frac{1}{n}\,\Psi_n(q_\theta,\,q_{S^n})+ \frac{c\,\varepsilon_n^2}{2} + \frac{\log 2}{n}.
\end{aligned}
\end{equation}
\end{theorem}

The first term on the l.h.s.~of inequality~\eqref{Eqn:RegVOI} relates the variational complementary probability $\widehat{Q}_\theta(\m F_{n,\varepsilon}^c)$ to the prior complementary probability $P_\theta(\m F_{n,\varepsilon}^c)$. As a consequence, an upper bound of this term controls the remainder variational probability mass outside the sieve $\m F_{n,\varepsilon}$. The second term $\int_{\theta\in\m F_{n,\varepsilon},\, r(\theta,\,\theta^\ast) \geq \varepsilon^2}$ $r(\theta,\,\theta^\ast)\, \qhat_\theta(\theta)\,d\theta$ in~\eqref{Eqn:RegVOI} is the variational Bayes risk over to the intersection between $\m F_{n,\varepsilon}$ and the set of all $\theta$ such that the loss $r(\theta,\,\theta^\ast)$ is at least $\varepsilon^2$. 

In \cite{pati2017statistical}, we proved a risk bound for the $\alpha = 1$ case under the much stronger assumption of a compact parameter space and the existence of a global test $\phi_n$ with type-I and II error rates bounded above by $e^{- C n \varepsilon_n^2}$. Under those assumptions, the result in \cite{pati2017statistical} can be recovered from our more general result in Theorem \ref{Thm:RegularPosterior} by setting $\m F_{n, \varepsilon} = \Theta$, and $\phi_{n, \varepsilon} = \phi_n$; the global test, for all $\epsilon$. Such stronger assumptions usually hold when the parameter space $\Theta$ is a compact subset of the Euclidean space --- however, in other cases such as unbounded parameter spaces or infinite dimensional functional spaces, such a global test function $\phi_n$ may not exist, signifying the necessity of Theorem \ref{Thm:RegularPosterior}. We also point out the preprint \cite{zhang2017convergence}, which appeared while this manuscript was in revision, where they consider the usual VB and their analysis is based on a direct application of the variational lemma in the proof of Theorem~\ref{thm:main} in the supplementary document. However, their results require a stronger prior concentration condition and their analysis does not involve latent variable models.

Similar to the development for $\alpha<1$, we can further simplify $\Psi_n$ by introducing more assumptions. Due to the space constraint, we only provide a counterpart of Theorem~\ref{thm:OI_mixture} under the assumptions made therein. Recall the definition of two KL neighbourhoods $\m B_n(\pi^\ast,\,\varepsilon)$ and $\m B_n(\mu^\ast,\,\varepsilon)$ defined in~\eqref{Eqn:KLN}.

\paragraph{Assumption P (Prior concentration):}
There exists some constant $C>0$ such that
\begin{align*}
P_\theta\big(\m B_n(\pi^\ast,\,\varepsilon_n)\big) \geq \exp\big(-C\,n\,\varepsilon_n^2)\quad\mbox{and}\quad P_\theta\big(\m B_n(\mu^\ast,\,\varepsilon_n)\big) \geq \exp\big(-C\,n\,\varepsilon_n^2).
\end{align*}

Under Assumptions T and P, Theorem~\ref{Thm:RegularPosterior} leads to a high probability bound on the over variational Bayes risk for loss $r(\theta,\,\theta^\ast)$, as summarized in the following Theorem.

\begin{theorem}[Parameter restriction-free]\label{Thm::RegularPosterior}
Under Assumptions T and P, it holds with $\bbP_{\theta^\ast}^{(n)}$ probability at least $1 - 3/\{(D-1)^2 \,n \, \varepsilon_n^2\}$ that for any $\varepsilon\in [\varepsilon_n,\,e^{c'\,n\varepsilon_n^2}]$ (for some constant $c'>0$), 
\begin{align*}
&\widehat{Q}_\theta\Big(r(\theta,\,\theta^\ast) \geq \varepsilon^2\Big) \leq  C_1\,\frac{\varepsilon_n^2}{\varepsilon^2}.
\end{align*}
In particular, this implies for any $R<e^{2c'\,n\varepsilon_n^2}$,
\begin{align*}
 \int_{\theta:\, r(\theta,\,\theta^\ast) \leq R} r(\theta,\,\theta^\ast)\, \qhat_\theta(\theta)\,d\theta \leq C_3\, \varepsilon_n^2\, \big(1+\log (R/\varepsilon_n)\big).
\end{align*}
\end{theorem}

In particular, if the sequence $\{\varepsilon_n:\,n\geq 1\}$ satisfies $n\,\varepsilon_n^2\to\infty$, then selecting $\varepsilon = M_n \,\varepsilon_n$ for $M_n\to \infty$ ($M_n\leq\varepsilon_n^{-1}$) leads to the asymptotic variational posterior concentration:
\begin{align*}
\widehat{Q}_\theta\Big(r(\theta,\,\theta^\ast) \leq M_n^2\, \varepsilon^2\Big) \to 1\quad\mbox{in probability,\ \  as $n\to\infty$}.
\end{align*}
The extra truncation $r(\theta,\,\theta^\ast) \leq R$ in the variational risk bound in the theorem is due to the quadratic decay of our upper bound to $\widehat{Q}_\theta\big(r(\theta,\,\theta^\ast) \geq \varepsilon^2\big)$. Since the risk upper bound only has a logarithmic dependence on the truncation level $R$, one can simply set it at a fixed large number to ensure an order $\m O(\varepsilon_n^2)$ risk bound in practice. In fact, this truncation can be eliminated under a stronger assumption (as in \cite{pati2017statistical}) that there is a global test function $\phi_{n}:\,\m Y^n\to[0,1]$, such that the type I error bound~\eqref{Eqn:T1} holds, and the following type II error bound holds for all $\theta\in \Theta$ satisfying $r(\theta,\,\theta^\ast) \geq \varepsilon_n^2$, 
\begin{align*}
\mb E_{\theta}[1- \phi_{n}]  \leq e^{-c\,n\,r(\theta,\,\theta^\ast)}.
\end{align*}
This can be seen from Theorem~\ref{Thm:RegularPosterior} by setting $\m F_{n}=\Theta$ and $\varepsilon=\varepsilon_n$ in inequality~\eqref{Eqn:RegVOI}, which implies 
\begin{align*}
c\,\int_{\Theta} r(\theta,\,\theta^\ast)\, \qhat_\theta(\theta)\,d\theta &\leq c\,\varepsilon_n^2 + c\,\int_{r(\theta,\,\theta^\ast) \geq \varepsilon_n^2} r(\theta,\,\theta^\ast)\, \qhat_\theta(\theta)\,d\theta \\
&\qquad\qquad\qquad\qquad\leq \frac{1}{n}\,\Psi_n(q_\theta,\,q_{S^n})+ \frac{3c\,\varepsilon_n^2}{2} + \frac{\log 2}{n}.
\end{align*}

\subsection{$\alpha$-VB using stronger divergences}\label{Sec::Stronger}
In this subsection, we consider an extension of our theoretical development for $\alpha$-VB where the KL divergence in the objective function is replaced by a stronger divergence $\bar{D}[p\,||\,q] \geq D[p\,||\,q]$, for example, $\chi^2$ divergence and R\'{e}nyi divergence \cite{Yingzhen2016}, and the corresponding variational approximation 
\begin{align*}
\bar{q}_{W^n} :\, = \argmin_{q_{W^n} \in \Gamma} \bar{D}\big[\,q_{W^n}(\cdot)\,\big|\big|\, p(\cdot\,|\, Y^n)\,\big].
\end{align*}
As another example, in some applications of variational inference \cite{zobay2014variational}, the minimization of the KL divergence over the variational density $q_{W^n}$ to the conditional density $p(W^n\,|\,Y^n)$ may not admit a closed-form updating formula, and some surrogate ELBO $\bar{L}(q_{W^n})$ as a lower bound to the ELBO $L(q_{W^n})$ is employed. Under the perspective of ELBO decomposition~\eqref{eq:elbo_dec}, this replacement is equivalent to using a stronger metric
\begin{align*}
\bar{D}\big[q_{W^n}\,\big|\big|\,p(\cdot\,|\,Y^n)\big] :\,= \log p(Y^n) - \bar{L}(q_{W^n}) \geq D\big[q_{W^n}\,\big|\big|\,p(\cdot\,|\,Y^n)\big].
\end{align*}

The following theorem provides a variational Bayes risk upper bound to $\bar{q}_{\theta}$. To simplify the presentation, the theorem is stated for the $\alpha<1$ case, although extension to $\alpha=1$ is straightforward. Define the equivalent objective function
\begin{align*}
\bar{\Psi}_\alpha(q_\theta,\,q_{S^n}) = \Psi_{n,\alpha}(q_\theta,\,q_{S^n}) + \Big(\bar{D}\big[q_{W^n}\,\big|\big|\,p(\cdot\,|\,Y^n)\big] - D\big[q_{W^n}\,\big|\big|\,p(\cdot\,|\,Y^n)\big]\Big),
\end{align*}
and the corresponding $\alpha$-VB solution $\bar{q}_{\theta,\alpha} = \argmin_{q_{W^n} \in \Gamma} \bar{\Psi}_\alpha(q_\theta,\,q_{S^n})$.
When $\bar{D}$ is the KL divergence $D$, objective function $\bar{\Psi}_\alpha$ reduces to the $\Psi_{n,\alpha}$ in \eqref{eq:VBapp1}.

\begin{theorem}\label{Thm::StrongerMetric}
For any $\zeta \in (0, 1)$, it holds with $\mb P_{\theta^\ast}^n$ probability at least $(1-\zeta)$ that for any probability measure $q_\theta \in\Gamma_{\theta}$ with $q_{\theta} \ll p_{\theta}$ and any probability measure $q_{S^n} \in \Gamma_{S^n}$ on $S^n$, 
\begin{align*}
\int \frac{1}{n}\,D^{(n)}_{\alpha}(\theta,\,\theta^\ast) \ \bar{q}_{\theta,\alpha}(d\theta)
\le \frac{\alpha}{n(1-\alpha)} \bar{\Psi}_{\alpha}(q_{\theta},\, q_{S^n})  + \frac{1}{n(1-\alpha)}  \log(1/\zeta).
\end{align*}
\end{theorem}
This theorem provides a simple replacement rule for $\alpha$-VB---if the $\alpha$-VB objective function $\Psi_{n,\alpha}$ is replaced with a upper bound $\bar{\Psi}_\alpha$, then a variational Bayes risk bound obtained by replacing $\Psi_{n,\alpha}$ with the upper bound $\bar{\Psi}_\alpha$ holds. We will apply this replacement rule to obtain a minimax variational risk bound for the mixture of Gaussian approximation in Section 4.


\section{Applications}

In this section, we apply our theory in Section 3 to concrete examples: mean field approximation to (low) high-dimensional Bayesian linear regression, (mixture of) Gaussian approximation to regular parametric models, mean field approximation to Gaussian mixture models, and mean field approximation to latent Dirichlet allocation. To simplify the presentation, all results are stated for $\alpha$-VB with $\alpha < 1$ and the $\alpha$ subscript in $\qhat_{\theta, \alpha}$ is dropped. Extensions to the $\alpha=1$ case are discussed in the supplement.

\paragraph{{\bf Example:} (Mean field approximation to low-dimensional Bayesian linear regression)}

Consider the following Bayesian linear model
\begin{align}\label{Eqn:BLM}
Y^n = X\beta + w,\quad w\sim\m N(0, \sigma^2I_n),
\end{align}
where $Y^n\in\mb R^n$ is the $n$-dim response vector, $X\in\mb R^{n\times d}$ the design matrix, $\beta\in\mb R^d$ the unknown regression coefficient vector of interest, and $\sigma$ the noise level. In this example, we consider the low-dimensional regime where $d\ll n$, and focus on independent prior $p_\beta\otimes p_\sigma$ for parameter pair $\theta=(\beta,\,\sigma)$ for technical convenience (the result also applies to non-independent priors).

We apply the mean-field approximation by using the following variational family
\begin{align*}
q(\beta,\, \sigma) =q_\beta(\beta)\, q_\sigma(\sigma)
\end{align*}
to approximate the joint $\alpha$-fractional posterior distribution of $\theta=(\beta,\,\sigma)$ with $\qhat_\theta=\qhat_\beta\otimes \qhat_\sigma$. This falls into our framework when there is no latent variable and $W^n=\theta$. 
Computational-wise, a normal prior for $\theta$ and an inverse gamma prior for $\sigma^2$ are attractive since they are ``conjugate" priors --- the resulting variational densities $\qhat_\beta$ and $\qhat_\sigma$ still fall into the same parametric families. An application of Theorem~\ref{thm:NoLatent} leads to the following result.

\begin{corollary}\label{coro:BLM}
Assume that the prior density is continuous, and thick around the truth $\theta^\ast=(\beta^\ast,\,\sigma^\ast)$, that is, $p_\theta(\theta^\ast)>0$ and $p_\sigma(\sigma^\ast)>0$. If $d/n\to 0$ as $n\to\infty$, then with probability tending to one as $n\to\infty$,
\begin{align*}
& \bigg\{\int h^2\big[p(\cdot\,|\,\theta)\,\big|\big|\, p(\cdot\,|\, \theta^\ast)\big] \ \qhat_\theta(\theta)\,d\theta\bigg\}^{1/2} \lesssim \sqrt{\frac{d}{n\,\min\{\alpha,\,1-\alpha\}}\,\log (d\,n)}.
\end{align*}
\end{corollary}
The convergence rate of $\m O(\sqrt{n^{-1}\,d\,\log (dn)})$ under the Hellinger distance implies that the $\alpha$-VB estimator $\widehat{\beta}_{\VB,\alpha}=\int \beta\,  \qhat_{\beta}(\beta)\,d\beta$ converges towards $\beta^\ast$ relative to the $\ell_2$ norm at rate $\sqrt{n^{-1}\,d\,\log (dn)}$ (under the condition that $n^{-1}\,X^TX$ has minimal eigenvalue bounded away from zero), which is the minimax rate up to logarithm factors. A similar $n^{-1/2}$ convergence rate has been obtained in \cite{you2014variational} by directly analyzing the stationary point of an alternating minimization algorithm. However, their analysis requires the closed-form updating formula based on a conjugate normal prior for $\beta$ and an inverse gamma prior for $\sigma^2$, and may not be applicable to other priors. On the other hand, Corollary \ref{coro:BLM} only requires the minimal conditions of prior thickness and continuity.

\paragraph{{\bf Example:} (Mean field approximation to high-dimensional Bayesian linear regression with spike-and-slab priors)}

In this example we continue to consider the Bayesian linear model~\eqref{Eqn:BLM}, but we are interested in the high-dimensional regime where $d\gg n$. Following standard practice to make sparsity assumptions in the $d \gg n$ regime, let $s\ll n$ denote the sparsity level, i.e., the number of non-zero coefficients, of the true regression parameter $\beta^\ast$.


We consider the popularly used spike-and-slab priors \cite{george1993variable} on $\beta$. Following \cite{george1993variable}, we introduce a latent indicator variable $z_j=I(\beta_j\neq 0)$ for each $\beta_j$ to indicate whether the $j$th covariate $X_j$ is included in the model, and call $z=(z_1,\ldots,z_d)\in\{0,1\}^d$ the latent indicator vector. We use the notation $\beta_z$ to denote the vector of nonzero components of $\beta$ selected by $z$, that is $\beta_z = (\beta_j : z_j = 1)$. Consider the following sparsity inducing hierarchical prior $p_{\beta,\,z}$ over $(\beta,\,z)$:
\begin{equation}\label{Eqn:SpikeSlab}
\begin{aligned}
&z_j \overset{iid}{\sim} \frac{1}{d} \,\delta_1+ \Big(1-\frac{1}{d}\Big)\,\delta_0, \quad j=1,\ldots,d,\\
&\beta_z\,|\,z \sim p_{\beta\,|\,z}, \quad \mbox{and}\quad \sigma\sim p_{\sigma},
\end{aligned}
\end{equation}
where the prior probability of $\{z_j=1\}$ is chosen as $d^{-1}$ so that on an average only $\m O(1)$ covariates are included in the model.  Let $z^\ast$ denote the indicator vector associated with the truth $\beta^\ast$.

By viewing the latent variable indicator vector $z$ as a parameter, we apply the block mean-field approximation \cite{Peter2012} by using the family 
\begin{align*}
q(\beta,\, \sigma,\, z) = q_\sigma(\sigma)\,\prod_{j=1}^d q_{z_j,\beta_j}(z_j,\beta_j)
\end{align*}
to approximate the joint $\alpha$-fractional posterior distribution of $\theta=(\beta,\,\sigma,\,z)$ with $\qhat_\theta(\theta)=\qhat_\sigma(\sigma)\,\prod_{j=1}^d \qhat_{z_j,\beta_j}(z_j,\beta_j)$. Although we have a high-dimensional latent variable vector $z$, the latent variable is associated with the parameter $\beta$, and not with the observation $Y^n$. 
Consequently, this variational approximation still falls into our framework without latent variable, that is, $W^n=\theta = (z,\,\beta)$ and $\Delta_J \equiv 0$. 
It turns out that the spike and slab prior with Gaussian slab is particularly convenient for computation --- it is ``conjugate" in that the resulting variational approximation falls into the same spike and slab family \cite{Peter2012}. 
An application of Theorem~\ref{thm:NoLatent} leads to the following result.

\begin{corollary}\label{coro:HBLM}
Suppose $p_{\beta\,|\,z^\ast}$ is continuous and thick at $\beta_{z^\ast}^\ast$, and $p_\sigma$ is continuous and thick at $\sigma^\ast$. If $s\log d/n \to 0$ as $n\to\infty$, then it holds with probability tending to one as $n\to\infty$ that
\begin{align*}
& \bigg\{\int h^2\big[p(\cdot\,|\,\theta)\,\big|\big|\, p(\cdot\,|\, \theta^\ast)\big] \ \qhat_\theta(\theta)\,d\theta\bigg\}^{1/2} \lesssim \sqrt{\frac{s}{n\,\min\{\alpha,\,1-\alpha\}}\,\log (d\,n)}.
\end{align*}
\end{corollary}
Corollary~\ref{coro:HBLM} implies a convergence rate $\sqrt{n^{-1}\,s\log (dn)}$ of the variational-Bayes estimator $\widehat{\beta}_{\VB,\alpha}$ under the restricted eigenvalue condition \cite{bickel2009simultaneous}, which is the minimax rate up to log terms for high-dimensional sparse linear regression. To our knowledge, \cite{ormerod2014variational} is the only literature that studies the mean-field approximation to high-dimensional Bayesian linear regression with spike and slab priors. They show estimation consistency by directly analyzing an iterative algorithm for solving the variational optimization problem with $\alpha=1$ and a specific prior. As before, Corollary~\ref{coro:HBLM} holds under very mild conditions on the prior and does not rely on having closed-form updates of any particular algorithm. 

Here, we considered the block mean-field instead of the full mean-field approximation which further decomposes $q_{z_j,\beta_j}$ into $q_{z_j}\otimes q_{\beta_j}$. In fact, the latter resembles a ridge regression estimator, and the KL term $\alpha^{-1}\,D(q_\theta\,||\,p_\theta)$ appearing in the upper bound in~\eqref{Eqn:key_var} cannot attain the minimax order $\sqrt{n^{-1}\,s\log d}$.

\paragraph{{\bf Example:}  (Gaussian approximation to regular parametric models)}
Consider a family of regular parametric models $\m P=\{\mb P^{(n)}_{\theta}:\,\theta\in\Theta\}$ where $n$ is the sample size, and the likelihood function $p^{(n)}_{\theta}$ is indexed by a parameter $\theta$ belonging to the parameter space $\Theta\subset \mb R^d$, which we assume to be compact. 
Let $p_\theta$ denote the prior density of over $\Theta$, and $Y^n=(Y_1,\ldots,Y_n)$ be the observations from $\mb P^{(n)}_{\theta^\ast}$, with $\theta^\ast$ being the truth. We apply the Gaussian approximation by using the Gaussian family $\Gamma_G$(restricted to $\Theta$)
\begin{align*}
q(\theta) \propto \m N(\theta;\, \mu, \Sigma)\, I_\Theta(\theta),\quad \mu\in\mb R^d \ \mbox{and}\  \Sigma\mbox{ is a $d\times d$ positive definite matrix}. 
\end{align*}
Details are postponed to the supplemental document for space constraints. 

\paragraph{{\bf Example:} (Mixture of Gaussian variational approximation to regular parametric models):}
Still consider the regular parametric model $\m P=\{\mb P_{\theta}^{(n)}:\,\theta\in\Theta\}$ as in the previous example. 
Now we consider the more flexible variational family $\Gamma_{MG}$ composed of 
\begin{align*}
q(\theta)= \sum_{j=1}^Jw_j\,\m N(\theta;\, \mu_j,\,\Sigma_j),\quad \sum_{j=1}^J w_j=1,\quad \mu_j\in\mb R^d,\quad \Sigma_j\in \mb R^{d\times d} \mbox{ is p.d.},
\end{align*}
as all mixtures of Gaussians with $J$ components, where $J$ is a pre-specified number. 
Let $q_j$ denote the $j$th component $\m N(\theta;\, \mu_j,\,\Sigma_j)$ of the variational density function $q$, and $\mb E_p$ denotes the expectation under a density function $p$. Since any probability distribution can be approximated by a mixture of Gaussians within arbitrarily small error with sufficient large number $J$ of components, this enlarged variational family may reduce the approximation error from using $\Gamma_G$ and become capable of capturing multimodality and heavy tail behaviour of the posterior \cite{zobay2014variational}. However, this additional flexibility in shape comes with the price of intractability of the entropy term $E_q[-\log q(\theta)]$. To facilitate computation, \cite{zobay2014variational} conducted an additional application of Jensen's inequality 
\begin{align*}
E_q[-\log q(\theta)] =\sum_{j=1}^J w_j\,\mb E_{q_j}[-\log q(\theta)] \geq  - \sum_{j=1}^J w_j\,\log \mb E_{q_j}[q(\theta)],
\end{align*}
yielding a lower bound to the ELBO as
\begin{align*}
L(q) =  E_q\bigg[\log \frac{p(Y^n,\,\theta)}{q(\theta)}\bigg] \geq \, \underbrace{E_q[\log p(Y^n,\,\theta)]-\sum_{j=1}^J w_j\,\log \mb E_{q_j}[q(\theta)]}_{\bar{L}(q)}.
\end{align*}
$\bar{L}(q)$ is more convenient to work with since its second term admits a simple analytic form. Using such a surrogate ELBO  places us directly in the framework of \S~\ref{Sec::Stronger} and an application of Theorem~\ref{Thm::StrongerMetric} leads to the following bound. 


\begin{corollary}\label{coro:MGVAP}
For any measure $q_\theta=\sum_{j=1}^J w_j\, q_j \in\Gamma_{MG}$, it holds with probability at least $1-\varepsilon$ that
\begin{equation}\label{Eqn:MGVAP}
\begin{aligned}
 &\int \Big\{\frac{1}{n} D^{(n)}_{\alpha}(\theta, \theta^\ast) \Big\}\, \qhat_\theta(\theta)\,d\theta \le -\frac{\alpha}{n (1 - \alpha)}\,\int_\Theta q_\theta(\theta) \,\log\frac{p(Y^n\,|\,\theta)}{p(Y^n\,|\,\theta^\ast)}\,d\theta \\
 &\qquad+ \frac{1}{n(1-\alpha)}\, \sum_{j=1}^J w_j\,\Big(\log \mb E_{q_j}[q_\theta(\theta)]-\mb E_{q_j}[\log p_\theta(\theta)]\Big) + \frac{1}{n(1-\alpha)}  \log(1/\varepsilon).
\end{aligned}
\end{equation}
In particular, under Assumption P, it holds with probability tending to one as $n\to\infty$ that
\begin{align*}
& \bigg\{\int h^2\big[p(\cdot\,|\,\theta)\,\big|\big|\, p(\cdot\,|\, \theta^\ast)\big] \ \qhat_\theta(\theta)\,d\theta\bigg\}^{1/2} \lesssim \sqrt{\frac{d}{n\,\min\{\alpha,\,1-\alpha\}}\,\log (d\,n)}.
\end{align*}
\end{corollary}
As a concrete example of the development in Section~\ref{Sec::Stronger}, this corollary suggests that the additional Jensen gap due to the $E_q[-\log q(\theta)]$ term  is reflected in the new variational inequality~\eqref{Eqn:MGVAP}. More precisely, the KL divergence term $D[q_\theta\,||\, p_\theta]$ is replaced by its upper bound $\sum_{j=1}^J w_j \,\big(\log \mb E_{q_j}[q_\theta(\theta)]-\mb E_{q_j}[\log p_\theta(\theta)]\big)$, which can be bounded by reducing it into a single Gaussian component case ($w_1=1$, and $w_j=0$ for $j=2,\,\ldots, J$).

\paragraph{{\bf Example:} (Mean field approximation to Gaussian mixture model)}

Suppose the true data generating model is the $d$-dimensional Gaussian mixture model with $K$ components,
\begin{align*}
Y \sim \sum_{k=1}^K \pi_k \, \m N(\mu_k,\,  I_d),
\end{align*}
where $\mu_k\in\mb R^d$ is the mean vector associated with the $k$th component and $\pi=(\pi_1,\ldots,\pi_K)\in\m S_K$ is the mixing probability. Here, for simplicity we assume the covariance matrix of each Gaussian component to be $I_d$. $\mu=(\mu_1,\ldots,\mu_K)$ and $\pi$ together forms the parameter $\theta=(\mu,\,\pi)$ of interest. 
By data augmentation, we can rewrite the model into the following hierarchical form by introducing the latent class variable $S$,
\begin{align*}
S\sim \mbox{Categorical}(\pi_1,\,\pi_2,\ldots,\, \pi_K),\quad Y\,|\,S=s \sim \m N(\mu_s,\,  I_d).
\end{align*}
Let $Y^n=(Y_1,\ldots,Y_n)$ be $n$ i.i.d.~ copies of $Y$ with parameter $\theta^\ast=(\mu^\ast,\,\pi^\ast)$, and $S^n=(S_1,\ldots,S_n)\in\{1,\ldots,K\}^n$ denote the corresponding latent variables. For simplicity, we assume that independent prior $p_\mu\otimes p_\pi$ are specified for $(\mu,\,\pi)$.

We apply the mean field approximation by using the family of density functions of the form
\begin{align*}
q(\pi, \mu, S^n) = q_\pi(\pi)\, q_\mu(\mu)\, q_{S^n}(s^n) = q_\pi(\pi)\, q_\mu(\mu)\, \prod_{i=1}^n q_{S_i}(s_i)
\end{align*}
to approximate the joint $\alpha$-fractional posterior distribution of $(\pi,\,\mu,\, S^n)$, producing the $\alpha$-mean-field approximation $\qhat_\theta\otimes \qhat_{S^n}$, where $(\qhat_\theta,\,\qhat_{S^n})$ are defined in \eqref{eq:alpha_VB_sol}. This variational approximation fits into the framework of Theorem~\ref{thm:OI_mixture}. Therefore, an application of this theorem leads to the following result.

\paragraph{{\bf Assumption R}: (regularity condition)}  There exists some constant $\delta_0>0$, such that each component of $\pi^\ast\in\m S_K$ is at least $\delta_0$.

\begin{corollary}\label{coro:GM}
Suppose Assumption R holds, and the prior densities $p_\mu$ and $p_\pi$ are thick and continuous at $\mu^\ast$ and $\pi^\ast$ respectively. If $d\,K/n\to 0$ as $n\to \infty$, then it holds with probability tending to one as $n\to\infty$ that
\begin{align*}
& \bigg\{\int h^2\big[p(\cdot\,|\,\theta)\,\big|\big|\, p(\cdot\,|\, \theta^\ast)\big] \ \qhat_\theta(\theta)\,d\theta\bigg\}^{1/2} \lesssim \sqrt{\frac{d\,K}{n\,\min\{\alpha,\,1-\alpha\}}\,\log (d\,n)}.
\end{align*}
\end{corollary}

As a related result, \cite{wang2006convergence} show that the with proper initialization, the coordinate descent algorithm for solving the variational optimization problem~\eqref{eq:alpha_VB_sol} with $\alpha=1$ under conjugate priors converges to a local minimum that is $\m O(n^{-1})$ away from the maximum likelihood estimate of $(\mu,\,\pi)$ by directly analyzing the algorithm using the contraction mapping theorem. In contrast, our proof does not require any structural assumptions on the priors, and can be easily extended to a broader class of mixture models beyond Gaussians.

\paragraph{{\bf Example:} (Mean field approximation to latent Dirichlet allocation)} 
As our final example, we consider Latent Dirichlet allocation (LDA, \cite{blei2003latent}), a conditionally conjugate probabilistic topic model~\cite{blei2012probabilistic} for uncovering
the latent ``topics'' contained in a collection of documents.
LDA treats documents as containing multiple topics, where a topic is a
distribution over words in a vocabulary.
Following the notation of \cite{hoffman2013stochastic}, let $K$ be a specific number of
topics and $V$ the size of
the vocabulary. LDA defines the following generative process:
\begin{enumerate}
  \item For each topic in $k=1,\ldots,K$,
  \begin{enumerate}
    \item draw a distribution over words $\beta_k \sim \mbox{Dir}_V(\eta_\beta)$.
  \end{enumerate}
  \item For each document in $d=1,\ldots,D$,
  \begin{enumerate}
    \item draw a vector of topic proportions $\gamma_d \sim \mbox{Dir}_K(\eta_\gamma)$.
    \item For each word in $n=1,\ldots,N$,
    \begin{enumerate}
      \item draw a topic assignment $z_{dn} \sim  \mbox{multi}(\gamma_d)$, then
      \item draw a word $w_{dn} \sim \mbox{multi}(\beta_{z_{dn}})$.
    \end{enumerate}
  \end{enumerate}
\end{enumerate}
Here $\eta_\beta \in \mb R_{+}$ is a hyper-parameter of the symmetric
Dirichlet prior on the topics $\beta$, and $\eta_\gamma \in \mb R_{+}^K$
are hyper-parameters of the Dirichlet prior on the topic proportions for each
document. $z_{dn}\in\{1,\ldots,K\}$ is the latent class variable over topics where
$z_{dn} = k$ indicates the $n$th word in document $d$ is assigned to the
$k$th topic. Similarly, $w_{dn}\in\{1,\ldots,V\}$ is the latent class variable over the words in the vocabulary where $w_{dn} = v$ indicates
that the $n$th word in document $d$ is the $v$th word in the vocabulary.  To facilitate adaptation to sparsity using Dirichlet distributions when $V,\, K\gg1$, we choose $\eta_\beta=1/V^{c}$ and $\eta_\gamma=1/K^{c}$ for some fixed number $c>1$ \cite{yang2014minimax}.

To apply our theory, we first identify all components in the model. For simplicity, we view $N$ as the sample size, and $D$ as the ``dimension" of the parameters in the model.
Under our vanilla notation, we are interested in learning parameters $\theta=(\pi,\,\mu)$, with $\pi=\{\gamma_d:\,d=1,\ldots,D\}$ and $\mu=\{\beta_k:\,k=1,\ldots, K\}$, from the posterior distribution $P(\pi,\, \mu,\,z\,|\, Y^n)$, where $S^{N}=\{S_n:\,n=1,\ldots,N\}$ with $S_n=\{z_{dn}:\, d=1,\ldots, D\}$ are latent variables, and $Y^{N}=\{Y_n:\,n=1,\ldots,N\}$ with $Y_n=\{w_{dn}:\, d=1,\ldots, D\}$ are the data, and the priors for $(\pi,\,\mu)$ are independent Dirichlet distributions $\mbox{Dir}_K(\eta_\gamma)$ and $\mbox{Dir}_V(\eta_\beta)$ whose densities are denoted by $p_{\pi}$ and $p_{\mu}$. The conditional distribution $p(Y^N\,|\,\mu,\,S^N)$ of the observation given the latent variable is 
\begin{align*}
\big(w_{dn}\,|\,\mu,\,z_{dn} \big) \sim \mbox{multi}(\beta_{z_{dn}}), \quad d=1,\ldots,D\ \mbox{and} \ n=1,\ldots,N.
\end{align*}
Finally, the $\alpha$-mean-field approximation considers using the family of probability density functions of forms
\begin{align*}
  q(\mu, \pi, S^N)
  &=q_\pi(\pi) \, q_\mu(\mu)\, \prod_{n=1}^Nq_{S_n}(S_n)=
  \prod_{k=1}^K q_{\beta_k}(\beta_k)
  \prod_{d=1}^D
  \left(
  q_{\gamma_d}(\gamma_{d})
  \prod_{n=1}^{N}
  q_{z_{dn}}(z_{dn})
  \right)
\end{align*}
to approximate the joint $\alpha$-fractional posterior of $(\mu,\,\pi,\, S^N)$. Since for LDA, each observation $Y_n$ is composed of $D$ independent observations, it is natural to present the variational inequality with the original loss function $D_{\alpha}\big[p(\cdot\,|\,\theta)\,\big|\big|\,p(\cdot\,|\, \theta^\ast)\big]=\sum_{d=1}^D D_{\alpha}\big[p_d(\cdot\,|\,\theta)\,\big|\big|\,p_d(\cdot\,|\, \theta^\ast)\big]$ re-scaling by a factor of $D^{-1}$, where $p_d(\cdot\,|\,\theta)$ denotes the likelihood function of the $d$th observation $w_{dn}$ in $Y_n$.
We make the following assumption.

\paragraph{{\bf Assumption S}: (sparsity and regularity condition)} Suppose for each $k$, $\beta^\ast_k$ is $d_k\ll V$ sparse, and for each $d$, $\gamma^\ast_d$ is $e_d\ll K$ sparse. Moreover, there exists some constant $\delta_0>0$, such that each nonzero component of $\beta^\ast_k$ or $\gamma^\ast_d$ is at least $\delta_0$.

\begin{corollary}\label{coro:LDA}
Under Assumption S, it holds with probability at least $1-C/\big(N\,\sum_{d=1}^D \varepsilon_{\gamma_d}^2+N\,\sum_{k=1}^K \varepsilon_{\beta_k}^2 \big)$ that
\begin{align*}
& \int \Big\{D^{-1}\,D_{\alpha}\big[p(\cdot\,|\,\theta)\,\big|\big|\,p(\cdot\,|\, \theta^\ast)\big] \Big\}\,\qhat_\theta(\theta)\,d\theta\\
 &\qquad \lesssim \frac{\alpha}{1-\alpha} \,\bigg\{\frac{1}{D}\,\sum_{d=1}^D \varepsilon_{\gamma_d}^2+\frac{1}{D}\,\sum_{k=1}^K \varepsilon_{\beta_k}^2 \bigg\} + \frac{1}{N\,(1-\alpha)} \,\bigg\{\frac{1}{D}\,\sum_{d=1}^D e_d\,\log\frac{K}{\varepsilon_{\gamma_d}} + \frac{1}{D}\,\sum_{k=1}^K d_k\,\log\frac{V}{\varepsilon_{\beta_k}}\bigg\},
\end{align*}
for any $\varepsilon_\gamma=(\varepsilon_{\gamma_1},\ldots,\varepsilon_{\gamma_d})$ and $\varepsilon_{\beta}=(\varepsilon_{\beta_1},\ldots,\varepsilon_{\beta_K})$. Therefore, if $\big(\sum_{d=1}^De_d+\sum_{k=1}^K d_k\big)/(DN)\to 0$ as $N\to\infty$, then it holds with probability tending to one that as $N\to\infty$
\begin{align*}
& \bigg\{\int D^{-1}\,h^2\big[p(\cdot\,|\,\theta)\,\big|\big|\, p(\cdot\,|\, \theta^\ast)\big] \ \qhat_\theta(\theta)\,d\theta\bigg\}^{1/2} \\
&\qquad \lesssim \sqrt{ \frac{\sum_{d=1}^De_d}{DN\,\min\{\alpha,\,1-\alpha\}}\,\log(DKN) + \frac{\sum_{k=1}^K d_k}{DN\,\min\{\alpha,\,1-\alpha\}}\,\log(KVN)}.
\end{align*}
\end{corollary}

Corollary~\ref{coro:LDA} implies estimation consistency as long as the ``effective" dimensionality $\sum_{d=1}^De_d+\sum_{k=1}^K d_k$ of the model is $o(DN)$ as the ``effective sample size" $DN\to\infty$.
In addition, the upper bound depends only logarithmically on the vocabulary size $V$ due to the sparsity assumption.

\section{Discussion}

The primary motivation behind this work is to investigate whether point estimates obtained from mean-field or other variational approximations to a Bayesian posterior enjoy the same statistical accuracy as those obtained from the true posterior, and we answer the question in the affirmative for a wide range of statistical models. To that end, we have analyzed a class of variational objective functions indexed by a temperature parameter $\alpha \in (0, 1]$, with $\alpha = 1$ corresponding to the usual VB, and obtained risk bounds for the variational solution which can be used to show (near) minimax optimality of variational point estimates. Our theory was applied to a number of examples, including the mean-field approximation to Bayesian linear regression with and without variable selection, Gaussian mixture models, latent Dirichlet allocation, and (mixture of) Gaussian variational approximation in regular parameter models. This broader class of objective functions can be fitted in practice with no additional difficulty compared to the usual VB. Hence, the proposed framework leads to a class of efficient variational algorithms with statistical guarantees. 

The theory for the $\alpha < 1$ and the $\alpha = 1$ (usual VB) case lead to interesting contrasts. For $\alpha < 1$, a prior mass condition suffices to establish the risk bounds for the Hellinger (and more generally, R{\'e}nyi divergences). However, the $\alpha = 1$ case requires additional conditions to be verified. When all conditions are met, there is no difference in terms of the rate of convergence for $\alpha < 1$ versus $\alpha = 1$. Hence, from a practical standpoint, the procedure with $\alpha < 1$ leads to theoretical guarantees with verification of fewer conditions. A comparison of second-order properties is left as a topic for future research, as is extension to models with dependent latent variates.

\newpage
\appendix
\section{Convention}
As a convention, all equations defined in this supplementary document are numbered (S1), (S2), \ldots, while equations cited from the main document retain their numbering in the main document. Similar for theorems, corollaries, lemmas etc. 

In \S~S2, we provide an empirical study to compare the $\alpha$-VB approach for $\alpha < 1$ to the the usual VB in some of the models discussed in \S4. In \S~S3, we illustrate applying our theory for continuous latent variable models. \S~S4 contains the Gaussian approximation example whose details were skipped in \S4 of the main document. \S~S5 provides proofs of all theoretical results. 

\section{Numerical Examples}
In this section, we illustrate the $\alpha$-VB procedure through several representative simulation examples. Since the objective functions $\Psi(q_{\theta})$ and $\Psi(q_{\theta}, q_{S^n})$ differ from usual VB only through the presence of $\alpha$, standard coordinate ascent variational inference (CAVI) algorithms\cite{bishop2006pattern,blei2017variational} can be implemented with simple modifications in the iterative updates. We implemented $\alpha$-VB with different choices of $\alpha$ between $0.5$ and $1$ and the point estimates were fairly robust to the choice of $\alpha$. 

\subsection{Bayesian high-dimensional linear regression}
Consider sparsity inducing hierarchical prior $p_{\beta,\,z}$ over $(\beta,\,z)$ as $\prod_{j=1}^d p_{\beta_j, z_j}$ with 
\begin{eqnarray*}
p_{\beta_j,\,1}  = \m N (\beta_j; 0, \nu_1\sigma^2), \quad p_{\beta_j,\,0}  = \delta_0(\beta_j); \quad z_j \sim \mbox{Bernoulli}(1/d),
\end{eqnarray*}
where $\delta_0$ denotes the point mass measure at $0$.
Apply the variational  approximation by using the family
\begin{align*}
q(\beta,\, \sigma,\, z) =q_\sigma(\sigma)\,\prod_{j=1}^d  q_{z_j,\beta_j}(z_j,\beta_j)
\end{align*}
where $q_{z_j,\beta_j}(z_j,\beta_j) = \prod_{j=1}^d N(\beta_j; \mu_j, \sigma_j^2)^{z_j} \delta_0(\beta_j)^{1-z_j}  [q_{z_j}(1)]^{z_j}[q_{z_j}(0)]^{1-z_j}$.  
Let $\phi_j =  q_{z_j}(1)$ for $j=1, \ldots, p$.  

An implementation of the $\alpha$-VB algorithm 
for Bayesian high-dimensional linear regression ($\alpha$-VB-HDR) is described in Algorithm \ref{algo2} and follows the batch-wise variational Bayes algorithm in Algorithm 2 of 
\cite{huang2016variational}.  
We sample $n=100$ observations from the linear regression model with $d=500$, with the entries of the covariate matrix ${\bf X}$ sampled independently from $N(0, 1.5^2)$, and error standard deviation $\sigma=1$.  The first $4$ coefficients are non-zero and are set equal to $(5,-4,-3,2)$. Figure \ref{fig:main2}
illustrates the performance of $\alpha$-VB-HDR for different values of $\alpha$.  In all the cases, the convergence of ELBO occurs within less than 20 iterates. 

\begin{algorithm}[htp!]
  \caption{$\alpha$-VB-HDR}\label{algo2}
  Set $\tilde{\sigma}= \sigma/ \sqrt{\alpha}$. \\
  Initialize $(\mu_1, \ldots, \mu_d), (\sigma_1^2, \ldots, \sigma_d^2)$, $(\phi_1,\ldots, \phi_d) = (1, \ldots, 1)'$ and $\Phi=\mbox{Diag}(\phi_1,\ldots, \phi_d)$. \\
   \While{ELBO does not converge}{
  $(\mu_1, \ldots, \mu_d)' =  ({\bf X}'{\bf X} + \Phi/\nu_1)^{-1}{\bf X}'{\bf Y}$\\
     \For{$j=1, \ldots, d$}
     {\begin{eqnarray*}
     \frac{1}{2\sigma_j^2} &=&\frac{\mbox{Diag}({\bf X}'{\bf X})_j}{2\tilde{\sigma}^2} + \frac{\phi_j}{2\nu_1\tilde{\sigma}^2} \\
       \phi_j &=&  \mathrm{Logit}^{-1}\bigg\{ \mathrm{Logit}(1/d) + \frac{1}{2}\log \bigg(\frac{\sigma_j^2}{\nu_1\tilde{\sigma}^2}\bigg) + \frac{\mu_j^2}{2\sigma_j^2} \bigg\}
       \end{eqnarray*}
    }
    }
  
 \end{algorithm}
 \begin{figure}
  \begin{center}
 \begin{tabular}{|c|c|}

      \hline
      \includegraphics[width=60mm]{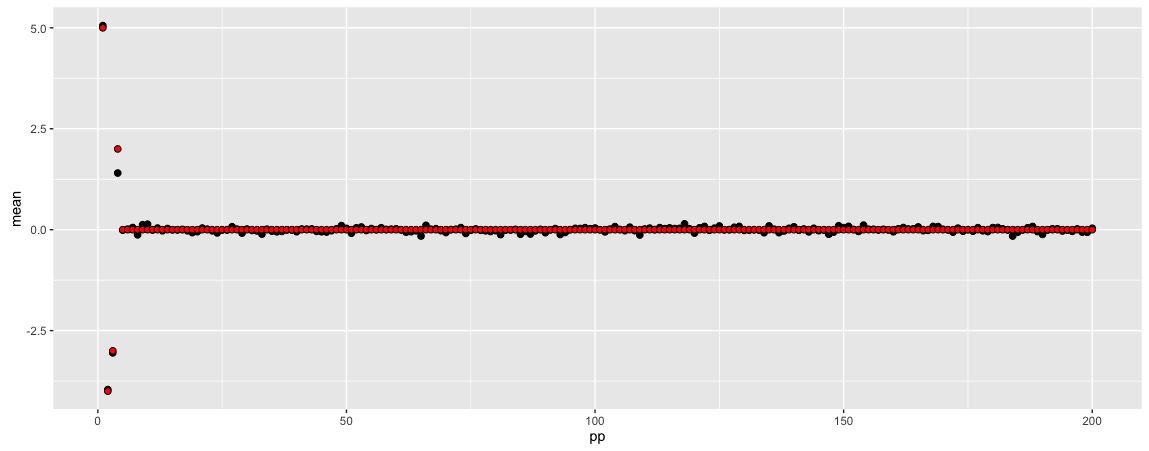} & \includegraphics[width=60mm]{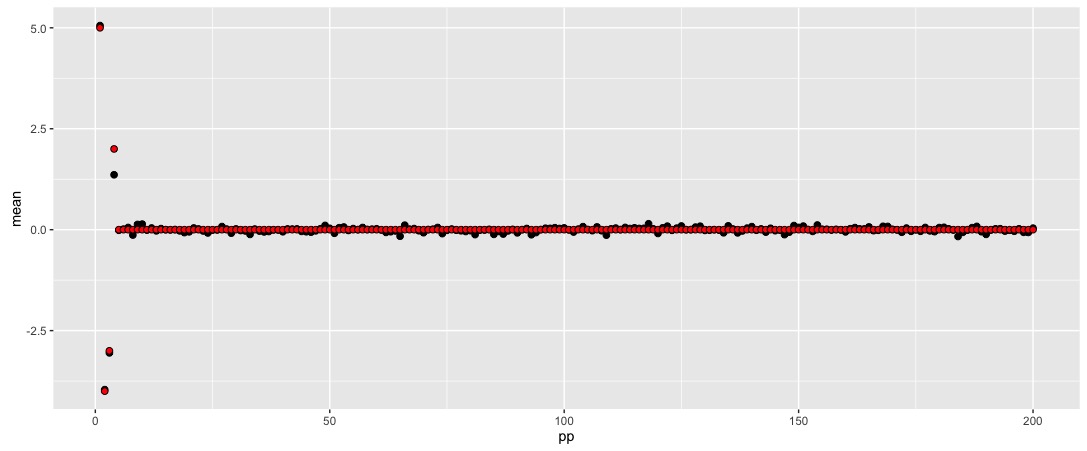} \\
      $\alpha = 0.5$ & $\alpha = 0.7$ \\ \hline
            \includegraphics[width=60mm]{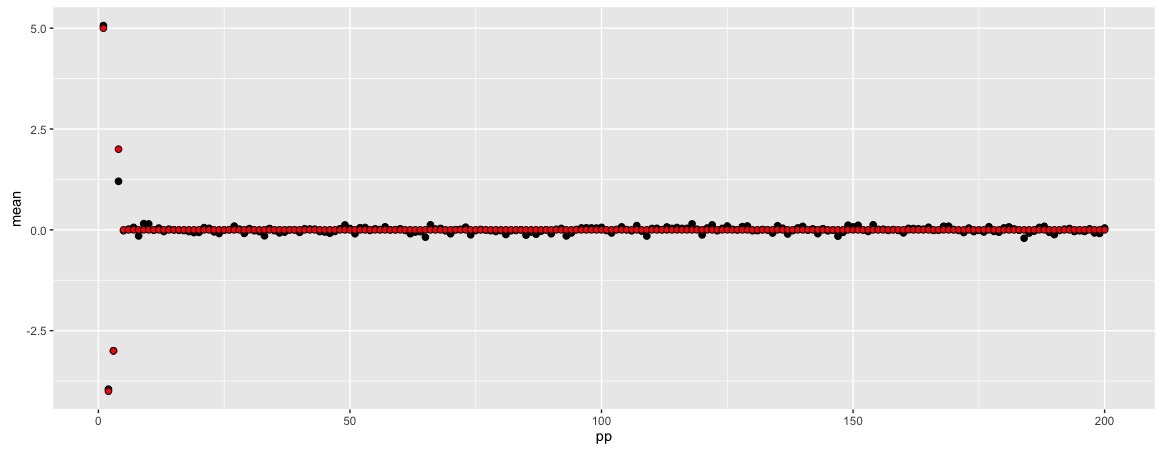} & \includegraphics[width=60mm]{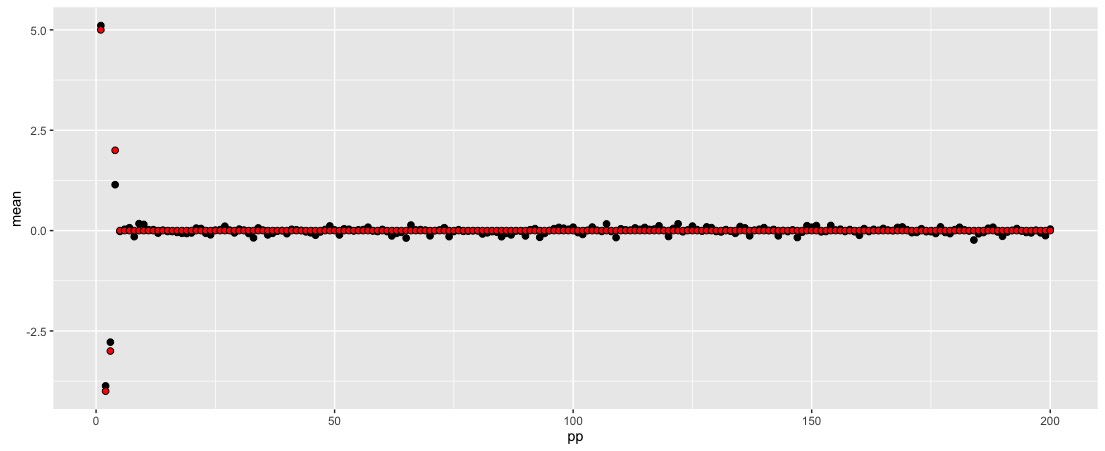} \\ 
             $\alpha = 0.95$ & $\alpha = 1$ \\ \hline
\end{tabular}
   \end{center}
   \caption{Plot of the $500$ coefficients;  true $\beta$ in  red and estimated means of $\beta$ using $\alpha$-VB-HDR in black.}
   \label{fig:main2}
   \end{figure}


\subsection{Gaussian mixture models}
We sample $n=1000$ bi-variate observations from 
\begin{align*}
Y \sim \sum_{k=1}^K \pi_k \, \m N(\mu_k,\,  I_2),
\end{align*}
for  $\pi_k =1/3$, $k=1, \ldots, K=3$.  $\mu_k$ are drawn from $\mbox{N}_2(0, 50I_2)$ for $k=1,\ldots, 3$.  Let $\pi(\mu_k) =  \mbox{N}_2(\mu_0, \sigma_0^2 I_2)$.  We use $\mu_0=(0, 0)'$ and $\sigma_0^2 = 50$.   For simplicity, we assume 
$\pi_k$ to be known in the study.  We apply the mean field approximation by using the family of density functions of the form
\begin{align*}
q(\mu, S^n) = q_\mu(\mu)\, q_{S^n}(s^n) = q_\mu(\mu)\, \prod_{i=1}^n q_{S_i}(s_i)
\end{align*}
Following \cite{blei2017variational}, we develop $\alpha$-VB algorithm for Gaussian mixture models ($\alpha$-VB-GMM), described in Algorithm \ref{algo1}. The derivation follows very closely to the case when $\alpha =1$ and hence the details are omitted. 
Numerical results are shown in Figure \ref{fig:main}.  In all the cases, the convergence of ELBO occurs within less than 10 iterates. 
 It is evident that for $\alpha$ close to $1$, $\alpha$-VB-GMM can recover the true density almost perfectly. 
 
\begin{algorithm}[htp!]
  \caption{$\alpha$-VB-GMM}\label{algo1}
  Initialize $\tilde{\mu}_k, \tilde{\sigma}_k, k=1,\ldots, K$ and $s_i, i=1, \ldots, n$. \\
  \While{\text{ELBO does not converge}}{
  \For{$i=1, \ldots, n$}
  { $q_{S_i}(s_i ) \propto \exp\{\alpha \log \pi_{s_i} + \alpha y_{s_i}E(\mu_{s_i}) - \alpha E(\mu_{s_i}^2/2)\}$} 

   \For{$k=1, \ldots, K$}
   {Update 
  \begin{eqnarray*}
  \tilde{\mu}_k =  \frac{\mu_0/\sigma_0^2 + \sum_{i=1}^kq_{S_i}(s_i=k) y_{s_i}}{1/\sigma_0^2 +  (1/\alpha)\sum_{i=1}^k q_{S_i}(s_i=k)}, \quad \tilde{\sigma}_k^2  =
   \frac{1}{1/\sigma_0^2 + (1/\alpha)\sum_{i=1}^k q_{S_i}(s_i=k)}  
  \end{eqnarray*}
  Set $q_{\mu_k} = \mbox{N}_2 (\mu_k; \tilde{\mu}_k,  \tilde{\sigma}_k^2 I_2)$}}
 
\end{algorithm}

    \begin{figure}[htp!]
        \begin{subfigure}[b]{0.475\textwidth}
            \centering
            \includegraphics[width=\textwidth]{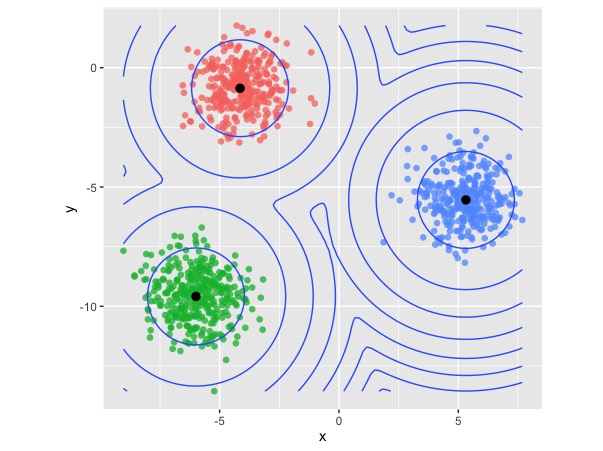}
            \caption[Network2]%
            {{\small True mixture density}}    
            \label{fig:1b}
        \end{subfigure}
        \hfill
        \begin{subfigure}[b]{0.475\textwidth}  
            \centering 
            \includegraphics[width=\textwidth]{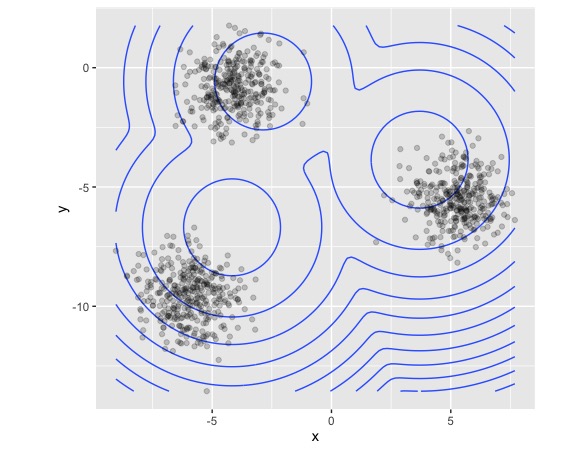}
            \caption[]%
            {$\alpha =0.7$}    
            \label{fig:2b}
        \end{subfigure}
        \vskip\baselineskip
        \begin{subfigure}[b]{0.475\textwidth}   
            \centering 
            \includegraphics[width=\textwidth]{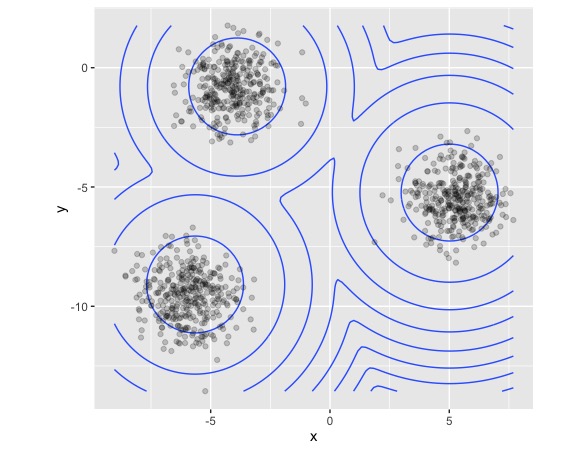}
            \caption[]%
            {$\alpha =0.95$}    
            \label{fig:3b}
        \end{subfigure}
        \quad
        \begin{subfigure}[b]{0.475\textwidth}   
            \centering 
            \includegraphics[width=\textwidth]{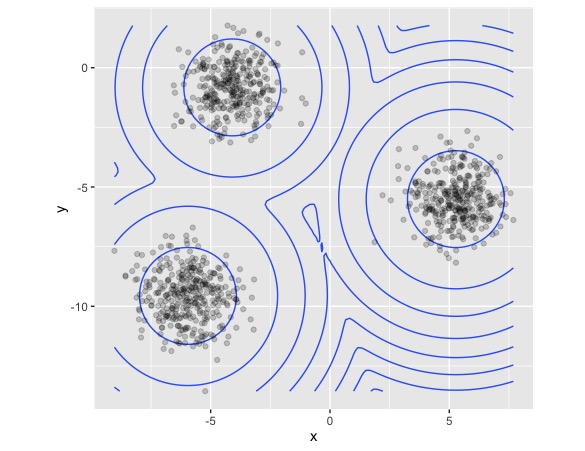}
            \caption[]%
            {$\alpha =1$}   
            \label{fig:4b}
        \end{subfigure}
        \caption[ ]
        {\small Contour plots for the true and predicted density using $\alpha$-VB-GMM. The colors in Figure \ref{fig:1b} represent different cluster components. } 
        \label{fig:main}
    \end{figure}

\subsection{Latent Dirichlet Allocation}
We implemented a version of LDA which is exactly same as Section 5.2 of \cite{blei2003latent}.  The approach is the same as the one described here with one minor difference.  The parameter $\eta_\gamma$  is set to $1/K$, but $\eta_\beta$ is estimated using an empirical Bayes approach described in Section 5.3 of \cite{blei2003latent} instead of fixing it to be $1/V$. To implement $\alpha$-VB, we note that the only change required will be to Equation (6) of Section 5.2 where we replace $\phi_{ni} \propto \beta_{iw_n} \exp \{E_q[\log \theta_i | \gamma] \}$ to $\phi_{ni} \propto \beta_{iw_n} \exp \{\alpha E_q[\log \theta_i | \gamma] \}$.   We provide an illustrative example of the use of an LDA model on a real data comprising of the first 5 out of 16,000 documents from a subset of the TREC AP corpus \cite{harmonoverview}. The maximum number of topics is set to $10$.    The top words from some of the resulting multinomial distributions $p(w|z)$ are illustrated in Table  \ref{tab:lda}. 
The distributions seem to capture some of the underlying topics in the corpus with decreasing word similarity  as $\alpha$ decreases.
\begin{table}[htp]
\tiny 
\caption{Top 5 words for each of the 10 extracted topics for $\alpha = 0.5, 0.7, 0.95, 1$ with the top 5 rows corresponding to $\alpha = 0.5$, next 5 rows corresponding to $\alpha = 0.7$ and so on.}
\scriptsize 
\begin{center}
\resizebox{\textwidth}{!}{%
\begin{tabular}{|cccccccccc|}
\hline 
T1 & T2 & T3 & T4 & T5 & T6 & T7 & T8 & T9 & T10 \\ \hline
	history & police & year & peres & liberace & school & classroom & i & peres & first \\ 
	ago & shot & people & israel & back & teacher & teacher & mrs & official & year \\
	york & gun & get & bechtel & mrs & guns & boy & jewelry & rappaport & minister \\ 
	president & students & first & offer & museum & boys & shot & museum & pipeline & new \\ 
	todays & door & just & memo & man & saturday & baptist & bloomberg & offer & invasion \\\hline
		history & police & year & peres & liberace & school & shot & i & peres & year \\ 
	president & students & people & offer & mrs & teacher & classroom & police & official & first \\ 
	ago & school & get & bechtel & bloomberg & guns & teacher & mrs & rappaport & new \\
	york & gun & volunteers & memo & back & shot & baptist & museum & pipeline & invasion \\
	todays & yearold & mail & israel & door & boys & marino & jewelry & offer & minister \\ \hline
	history & school & first & memo & liberace & first & shot & mrs & peres & people \\ 
	president & police & year & effect & door & year & baptist & i & offer & get \\ 
	ago & teacher & just & wage & back & day & marino & police & official & year \\ 
		first & students & died & quoted & mrs & died & teacher & museum & rappaport & thompson \\ 
	year & boys & day & bechtel & bloomberg & people & kids & bloomberg & bechtel & program \\ \hline
	ago & police & get & memo & liberace & teacher & shot & i & peres & people \\
	president & school & volunteers & bechtel & mrs & school & police & police & official & year \\ 
	history & students & year & peres & bloomberg & shot & baptist & mrs & offer & thompson \\ 
	first & teacher & mail & offer & back & guns & teacher & museum & rappaport & program \\ 
	year & boys & people & israel & door & students & classroom & jewelry & pipeline & get \\ \hline
\end{tabular}}
\end{center}
\label{tab:lda}
\end{table}%

\section{Extension to continuous latent variables}

As discussed in the main draft, we extend results on mean field approximations to models from discrete latent variables to continuous latent variables. For simplicity, we only focus on the $\alpha < 1$ case. 
All the proofs proceed in a similar way---the only difference is replacing all sums over latent variables with integrations. Specifically, in the definition of $\Psi_{\alpha}$ in \eqref{eq:VBapp1}, the only change takes place in the quantity $\Delta_J$, where the approximation to the likelihood is now made with continuous latent variables. 
We present a version where i.i.d.~copies $T^n=(T_1,\ldots,T_n)$ of the latent variable $T\in \m T$ are continuous and there is no restrictions on the variational factor $q_{T_i}$ for each latent variable $T_i$.  
In this setting, the $\alpha$-VB objective function is simplified to 
\begin{equation}\label{Eqn:MF_MLV_con}
\begin{aligned}
&\Psi_{\alpha}(q_\theta, q_{T^n}) = \\
&\qquad- \int_\Theta q_\theta(\theta)\,\sum_{i=1}^n \int_{\m T} q_{T_i}(t_i)\, \log\frac{p(Y_i\,|\,\mu,t_i) \, p_{T_i}(t_i\,|\,\pi)}{p(Y_i\,|\, \theta^\ast)\,q_{T_i}(t_i)}\,d t_i\, d\theta + \alpha^{-1} D(q_\theta\, ||\, p_\theta),
\end{aligned}
\end{equation}
where we assume that the distribution family $p_T(\cdot\,|\, \pi)$ for the latent variable is indexed by its own parameter $\pi$, and recall that $\mu$ is the parameter in the likelihood function $p(Y\,|\,\mu,\,T)$ of response $Y$ given the latent variable $T$, and $\theta=(\pi,\,\mu)$ are the parameters. 

Similar to the discrete case, for continuous latent variables, we define the following two KL neighborhoods of $\pi^\ast$ and $\mu^\ast$
\begin{align*}
\m B^{con}_n(\pi^\ast,\,\varepsilon_\pi) &= \Big\{D\big[p_T(\cdot\,|\,\pi^\ast)\,||\, p_T(\cdot\,|\,\pi)\big] \leq \varepsilon_\pi^2,\ \ V\big[p_T(\cdot\,|\,\pi^\ast)\,||\, p_T(\cdot\,|\,\pi)\big] \leq \varepsilon_\pi^2\Big\},\\
\m B^{con}_n(\mu^\ast,\,\varepsilon_\mu) &=  \Big\{\sup_{t}D\big[ p(\cdot\,|\,\mu^\ast,t)\, \big|\big| \, p(\cdot\,|\,\mu,t)\big]\leq \varepsilon_\mu^2,\\
&\qquad\qquad\qquad\qquad\qquad\qquad \sup_t V\big[ p(\cdot\,|\,\mu^\ast,t)\, \big|\big| \, p(\cdot\,|\,\mu,t)\big] \leq \varepsilon_\mu^2\Big\}.
\end{align*}
We now state a theorem with the same combined conclusions of Corollary \ref{coro:main} and Theorem \ref{thm:OI_mixture} for the continuous case. The proof is similar and hence omitted. 
\begin{theorem}\label{thm:OI_mixture_cont}
For any measure $q_\theta$ over $\theta$ satisfying $q_\theta\ll p_\theta$, it holds with probability at least $(1-\zeta)$ that
\begin{equation}\label{Eqn:key_var_cont}
\begin{aligned}
& \int \Big\{D_{\alpha}\big[p(\cdot\,|\,\theta)\,\big|\big|\, p(\cdot\,|\, \theta^\ast)\big] \Big\}\,\qhat_{\theta,\alpha}(\theta)\,d\theta\\
& \leq -\frac{\alpha}{n(1 - \alpha)} \int_\Theta  q_\theta(\theta)\, \sum_{i=1}^n \bigg\{\int_{\m T} \qtil_{T_i}(t_i)\,  \log\frac{p(Y_i\,|\,\mu,t_i) \, p_{T_i}(t_i\,|\,\pi)}{p(Y_i\,|\,\mu^\ast,t_i) \, p_{T_i}(t_i\,|\,\pi^\ast)}\, dt_i \bigg\}\, d\theta\\
&\qquad\qquad\qquad\qquad\qquad\qquad\qquad+ \frac{1}{n(1-\alpha)}\,D(q_\theta\, ||\, p_\theta)  + \frac{1}{n(1-\alpha)}  \log(1/\zeta),
\end{aligned}
\end{equation}
where $\qtil_{T_i}$ is a probability distribution over $\m T$ satisfying
\begin{align}\label{Eqn:q_T}
\qtil_{T_i}(t_i)= \frac{p_{T_i}(t_i\,|\,\pi^\ast)\, p(Y_i\,|\,\mu^\ast, \,t_i)}{p(Y_i\,|\,\theta^\ast)},\quad t_i\in\m T.
\end{align}
Moreover, for any fixed $(\varepsilon_\pi,\,\varepsilon_\mu) \in (0, 1)^2$, with $\bbP_{\theta^\ast}$ probability at least $1 - 5/\{(D-1)^2 \,n \, (\varepsilon_\pi^2+\varepsilon_\mu^2)\}$, it holds that
\begin{equation}\label{eq:OI_mixture_cont}
\begin{aligned}
 &\int \Big\{D_{\alpha}\big[p(\cdot\,|\,\theta)\,\big|\big|\, p(\cdot\,|\, \theta^\ast)\big] \Big\}\,\qhat_{\theta,\alpha}(\theta)\,d\theta \le  \frac{D\, \alpha}{1-\alpha} \, (\varepsilon_\pi^2+\varepsilon_\mu^2)\\
 &\   + \Big\{ - \frac{1}{n(1-\alpha)} \log P_\pi\big[\m B^{con}_n(\pi^\ast,\,\varepsilon_\pi)\big] \Big\} 
+ \Big\{ - \frac{1}{n(1-\alpha)} \log P_\mu\big[B^{con}_n(\mu^\ast,\,\varepsilon_\mu) \big] \Big\}.
\end{aligned}
\end{equation}
\end{theorem}

In presence of continuous latent variables, if the mean-field variational family is further constrained by restricting each factor $q_{T_i}$ corresponding to the latent variable $T_i$ to belong to a parametric family $\Gamma_{T_i}$, such as the exponential family, then the Bayes risk bound of Theorem~\ref{thm:OI_mixture_cont} still applies as long as the family $\Gamma_{T_i}$ for $q_{T_i}$ contains densities of form~\eqref{Eqn:q_T}---which is the case if the conditional distribution $p(T_i\,|\,\pi)$ also belongs to $\Gamma_{T_i}$ and the model $p(Y_i\,|\,\mu,\,\T_i)$ is conjugate with respect to family $\Gamma_{T_i}$.

\section{Gaussian approximation to regular parametric models} 
We discuss the details of this example from \S4 which were skipped in the main document. For sake of completeness, we remind the readers of the setting. 

Consider a family of regular parametric models $\m P=\{\mb P^{(n)}_{\theta}:\,\theta\in\Theta\}$ where $n$ is the sample size, and the likelihood function $p^{(n)}_{\theta}$ is indexed by a parameter $\theta$ belonging to the parameter space $\Theta\subset \mb R^d$, which we assume to be compact. 
Let $p_\theta$ denote the prior density of over $\Theta$, and $Y^n=(Y_1,\ldots,Y_n)$ be the observations from $\mb P^{(n)}_{\theta^\ast}$, with $\theta^\ast$ being the truth. We apply the Gaussian approximation by using the Gaussian family $\Gamma_G$(restricted to $\Theta$)
\begin{align*}
q(\theta) \propto \m N(\theta;\, \mu, \Sigma)\, I_\Theta(\theta),\quad \mu\in\mb R^d \ \mbox{and}\  \Sigma\mbox{ is a $d\times d$ positive definite matrix}. 
\end{align*}
The Gaussian variational approximation $\qhat_\theta$ as
\begin{align*}
\qhat_\theta:\,=\argmin_{q_\theta\in\Gamma_G}\bigg\{-\alpha\,\int_\Theta \int q_\theta(\theta)\,\log p^{(n)}_\theta(Y^n)\,d\theta  + D(q_\theta\,| |\, p_\theta)\bigg\}.
\end{align*}
we make the following assumption.

\vspace{1em}
\paragraph{{\bf Assumption P}: (prior thickness and regularity condition)}  The prior density $p_\theta$ satisfies $\inf_{\theta\in\Theta} p_\theta(\theta) >0$, and there exists some constant $C$ such that $D\big[p(\cdot\,|\,\theta_1)\,\big|\big|\,p(\cdot\,|\,\theta_2)\big] \leq C\,\|\theta_1-\theta_2\|^2$ and $V\big[p(\cdot\,|\,\theta_1)\,\big|\big|\,p(\cdot\,|\,\theta_2)\big] \leq C\,\|\theta_1-\theta_2\|^2$ holds for all $(\theta_1,\,\theta_2)\in\Theta^2$.

\begin{corollary}\label{coro:GApp}
Under Assumption P, it holds with probability tending to one as $n\to\infty$ that
\begin{align*}
& \bigg\{\int h^2\big[p(\cdot\,|\,\theta)\,\big|\big|\, p(\cdot\,|\, \theta^\ast)\big] \ \qhat_\theta(\theta)\,d\theta\bigg\}^{1/2} \lesssim \sqrt{\frac{d}{n\,\min\{\alpha,\,1-\alpha\}}\,\log (d\,n)}.
\end{align*}
\end{corollary}

Under the model identifiability condition $h^2\big[p(\cdot\,|\,\theta)\,\big|\big|\, p(\cdot\,|\, \theta^\ast)\big] \gtrsim \|\theta-\theta^\ast\|^2$, Corollary~\ref{coro:GApp} implies a convergence rate $\sqrt{n^{-1}\,d \log(dn)}$ for the variational-Bayes estimator $\widehat{\theta}_B$ of $\theta$.
By examining the proof of the corollary, we find that the normality form in the variational approximation does not play a critical role in the proof---similar Bayesian risk upper bounds hold under some additional conditions for a broader class of variational distributions as well, such as any location-scale distribution family with sub-exponential tails.
It is a well-known fact \cite{wang2005inadequacy,westling2015establishing} that the covariance matrices from the variational approximations are typically ``too small" compared with those for the sampling distribution of the maximum likelihood estimator, which combined with the Bernstein von-mises theorem implies that the variational approximation $\qhat_\theta$ may not converge to the true posterior distribution. 
This fact combined with the result in Corollary~\ref{coro:GApp} indicates: 1. minimizing the KL divergence over the variational family forces the variational distribution $\qhat_\theta$ to concentrate around the truth $\theta^\ast$ at the optimal rate (due to the heavy penalty on the tails in the KL divergence); 2. however, the local shape of $\qhat_\theta$ around $\theta^\ast$ can be far away from that of the true posterior due to dis-match between the distributions in the variational family and the true posterior.

\section{Proofs}\label{section:proof}
In this section, we present proofs of all technical results in the main document. 

\subsection{Proof of Theorem~\ref{thm:main}}
We first state a key variational lemma that plays a critical role in the proof. 
\begin{lemma}\label{lem:var}
Let $\mu_\theta$ be a probability measure over $\theta$ and $\mu_{S^n}$ be a probability measure over $S^n$, and $h(\theta,S^n)$ a measurable function such that for any fixed $S^n$, $e^{h(\cdot,S^n)} \in L_1(\mu_\theta)$. Then, 
\begin{align*}
\log \int \sum_{s^n} e^{h(\theta,s^n)}& \,\mu_{S^n}(s^n)\, \mu_\theta(d\theta) \\
&= \sup_{\rho_n(\theta,S^n)} \bigg[ \int \sum_{s^n} h(\theta,s^n)\, \rho_n(d\theta,s^n) - D(\rho_n(\theta,S^n) \,\vert \vert\, \mu_\theta\otimes\mu_{S^n}) \bigg],
\end{align*}
where the supremum is over all probability measures $\rho_n(\theta, S^n) \ll \mu_\theta\otimes\mu_{S^n}$. 
Further, the supremum on the right hand side is attained when 
$$
\frac{\rho_n(d\theta,s^n)}{\mu(d\theta)\,\mu_{S^n}(s^n)} = \frac{e^{h(\theta,S^n)}}{\int \sum_{s^n} e^{h(\theta,s^n)}\, \mu_{S^n}(s^n)\, \mu(d\theta)}.
$$
\end{lemma}
\begin{proof}
Use the well-known variational dual representation of the KL divergence (see, e.g., Corollary 4.15 of \cite{boucheron2013concentration}) which states that for any probability measure $\mu$ and any measurable function $h$ with $e^h \in L_1(\mu)$, one has 
$$
\log \int e^{h(\eta)} \mu(d\eta) = \sup_{\rho \ll \mu} \bigg[\int h(\eta) \rho(d\eta) - D(\rho\,\vert\vert\,\mu)\bigg],
$$
where the supremum is over all probability distributions $\rho \ll \mu$, and equality is attained when $d\rho/d\mu \propto e^h$. This fact simply follows upon an application of Jensen's inequality. In the current context, set $\eta = (\theta, s^n)$, $\mu = \mu_\theta \otimes \mu_{S^n}$ and $\rho(d\eta) = \rho_n(d\theta, s^n)$ to obtain the conclusion of Lemma \ref{lem:var}. 
\end{proof}

Return to the proof of the theorem. By applying Jensen's inequality to function $x \mapsto x^\alpha$ ($\alpha<1$), we obtain that, for any (possibly data dependent) measure $q_{S^n}$,
\begin{align*}
\bbE_{\theta^\ast} \bigg[ \sum_{s^n} q_{S^n}(s^n)&\, \exp\Big\{\alpha \,\log\frac{p(Y^n\,|\,\mu,s^n) \, \pi_{s^n}}{p(Y^n\,|\, \theta^\ast)\,q_{S^n}(s^n)}\Big\} \bigg] \\
 =&\, \int_{\mb R^n} \sum_{s^n} q_{S^n}(s^n)\, \bigg\{\frac{p(Y^n\,|\,\mu,s^n) \, \pi_{s^n}}{p(Y^n\,|\, \theta^\ast)\,q_{S^n}(s^n)}\bigg\}^\alpha \, p(Y^n\,|\,\theta^\ast) \,dY^n\\
\leq&\, \int_{\mb R^n}  \bigg\{\sum_{s^n} q_{S^n}(s^n)\,  \frac{p(Y^n\,|\,\mu,s^n) \, \pi_{s^n}}{p(Y^n\,|\, \theta^\ast)\,q_{S^n}(s^n)}\bigg\}^\alpha \, p(Y^n\,|\,\theta^\ast) \,dY^n\\
\le & \int \bigg\{ \frac{p(Y^n\,|\,\theta)}{p(Y^n\,|\,\theta^\ast)} \bigg\}^{\alpha}  \, p(Y^n\,|\,\theta^\ast) \,dY^n\\
=&\, e^{-(1-\alpha) \,D_{\alpha}^{(n)}(\theta, \theta^\ast)},
\end{align*}
with $D_{\alpha}^{(n)}(\theta, \theta^\ast)$ defined in the first display of \S3.1. 
Thus, for any $\zeta \in (0, 1)$, we have
\begin{align*}
\bbE_{\theta^\ast} \bigg[ \sum_{s^n} q_{S^n}(s^n) &\, \exp\Big\{\alpha \,\log\frac{p(Y^n\,|\,\mu,s^n) \, \pi_{s^n}}{p(Y^n\,|\, \theta^\ast)\,q_{S^n}(s^n)} \\
&+ (1 - \alpha) \, D_{\alpha}^{(n)}(\theta, \theta^\ast) - \log(1/\zeta) \Big\}\bigg] \le \zeta. 
\end{align*}
Integrating both side of this inequality with respect to $p_\theta$ and interchanging the integrals using Fubini's theorem, we obtain
\begin{align*}
\bbE_{\theta^\ast} \bigg[ \int_\Theta \sum_{s^n} p_\theta(\theta)\, q_{S^n}(s^n)&\, \exp\Big\{\alpha \,\log\frac{p(Y^n\,|\,\mu,s^n) \, \pi_{s^n}}{p(Y^n\,|\, \theta^\ast)\,q_{S^n}(s^n)} \\
&+ (1 - \alpha) \, D_{\alpha}^{(n)}(\theta, \theta^\ast) - \log(1/\zeta) \Big\} \, d\theta\bigg] \le \zeta. 
\end{align*}
Now, apply Lemma \ref{lem:var} with $\mu_\theta = p_\theta$, $\mu_{S^n} = q_{S^n}$ and 
$$
h(\theta, s^n) = \alpha \,\log\frac{p(Y^n\,|\,\mu,s^n) \, \pi_{s^n}}{p(Y^n\,|\, \theta^\ast)\,q_{S^n}(s^n)} + (1 - \alpha) \, D_{\alpha}^{(n)}(\theta, \theta^\ast) - \log(1/\zeta),
$$
to obtain that 
\begin{align*}
\bbE_{\theta^\ast}\exp \sup_{\rho(\theta,S^n)} \bigg[\int_\Theta \sum_{s^n}&\, \bigg\{ \alpha \, \log\frac{p(Y^n\,|\,\mu,s^n) \, \pi_{s^n}}{p(Y^n\,|\, \theta^\ast)\,q_{S^n}(s^n)} +  (1 - \alpha) \, D_{\alpha}^{(n)}(\theta, \theta^\ast)\\
&\qquad \qquad - \log(1/\zeta)\bigg\} \rho(d\theta,s^n) - D(\rho \,\vert \vert \,p_\theta \otimes q_{S^n})  \bigg] \le \zeta.
\end{align*}
If we choose $\rho =q_\theta\otimes q_{S^n}$ in the preceding display for any (possibly data dependent) $q_\theta \ll p_\theta$, then
\begin{align*}
\bbE_{\theta^\ast} \exp \bigg[\int_\Theta \sum_{s^n} \bigg\{ \alpha &\, \log\frac{p(Y^n\,|\,\mu,s^n) \, \pi_{s^n}}{p(Y^n\,|\, \theta^\ast)\,q_{S^n}(s^n)} +  (1 - \alpha) \, D_{\alpha}^{(n)}(\theta, \theta^\ast)\\
& \qquad \qquad\quad - \log(1/\zeta)\bigg\}\, q_\theta(d\theta)\, q_{S^n}(s^n) - D(q_\theta \,\vert \vert \,p_\theta)  \bigg] \le \zeta.
\end{align*}
By applying Markov's inequality, we further obtain that with $\bbP_{\theta^\ast}$ probability at least $(1 - \zeta)$, 
\begin{align*}
&\,(1 - \alpha) \int  D_{\alpha}^{(n)}(\theta, \theta^\ast) \,q_\theta(d\theta)\\
&\qquad\le - \alpha  \int_\Theta\sum_{s^n} \bigg\{\log\frac{p(Y^n\,|\,\mu,s^n) \, \pi_{s^n}}{p(Y^n\,|\, \theta^\ast)\,q_{S^n}(s^n)}\bigg\} \, q_\theta(d\theta)\,q_{S^n}(s^n)+ D(q_\theta \,\vert \vert \,p_\theta)  + \log(1/\zeta) \\
&\qquad= \alpha \, \Psi_\alpha(q_\theta,q_{S^n}) + \log(1/\zeta),
\end{align*}
since, from \eqref{eq:KL_decomp} -- \eqref{eq:KL_decomp1} and \eqref{eq:VBapp1}, 
$$
\Psi_\alpha(q_\theta, q_{S^n}) = -\sum_{s^n} \bigg[q_{S^n}(s^n) \log\frac{p(Y^n\,|\,\mu,s^n) \, \pi_{s^n}}{p(Y^n\,|\, \theta^\ast)\,q_{S^n}(s^n)}\bigg] \, q_\theta(d\theta) + \alpha^{-1} D(q_\theta \,\vert \vert \,p_\theta). 
$$
Since the inequality in the penultimate display holds for any (possibly data dependent) $q_\theta \ll p_\theta$ and $q_{S^n}$, we obtain, in particular, 
\begin{align*}
(1 - \alpha) \int  D_{\alpha}^{(n)}(\theta, \theta^\ast) \,\qhat_{\theta,\alpha}(d\theta)
\le \alpha \, \Psi_\alpha(\qhat_{\theta,\alpha},\qhat_{S^n,\alpha}) + \log(1/\zeta). 
\end{align*}
The conclusion of the Theorem follows since $\Psi_\alpha(\qhat_{\theta,\alpha},\qhat_{S^n,\alpha}) \le \Psi_\alpha(q_\theta, q_{S^n})$ for any $q_\theta \ll p_\theta$ and $q_{S^n}$.


\subsection{Proof of Theorem~\ref{thm:OI_mixture}}

We choose $q_\theta$ as the probability density function $q^\ast_\theta$ of 
$$
Q^\ast_\theta=\frac{P_\pi\big[\, \cdot\,  \cap\,  \m B_n(\pi^\ast,\,\varepsilon_\pi)\big]\otimes P_\mu\big[\, \cdot\,  \cap\,  \m B_n(\mu^\ast,\,\varepsilon_\mu)\big]}{P_\pi\big[ \m B_n(\pi^\ast,\,\varepsilon_\pi)\big]\cdot P_\mu\big[ \m B_n(\mu^\ast,\,\varepsilon_\mu)\big]},
$$
the product measure of restrictions of the priors $(P_\pi,\,P_\mu)$ for $(\pi,\,\mu)$ to two KL neighborhoods $\m B_n(\pi^\ast,\,\varepsilon_\pi)$ and $\m B_n(\mu^\ast,\,\varepsilon_\mu)$ around $(\pi^\ast,\,\mu^\ast)$.

Next, we will characterize the first two moments of the first term on the r.h.s.~in inequality~\eqref{Eqn:key_var} under this choice of $q^\ast_\theta$.
By applying Fubini's theorem, we have
\begin{align*}
&\bbE_{\theta^\ast}\bigg[\int_\Theta q^\ast_\theta(\theta)\, \sum_{i=1}^n \sum_{s_i}\qtil_{S_i}(s_i)\,  \log\frac{p(Y_i\,|\,\mu,s_i) \, \pi_{s_i}}{p(Y_i\,|\,\mu^\ast,s_i) \, \pi^\ast_{s_i}}\, d\theta\bigg] \\
&\qquad\qquad\qquad= \int_\Theta  \bbE_{\theta^\ast}\bigg[  \sum_{i=1}^n \sum_{s_i}\qtil_{S_i}(s_i)\,  \log\frac{p(Y_i\,|\,\mu,s_i) \, \pi_{s_i}}{p(Y_i\,|\,\mu^\ast,s_i) \, \pi^\ast_{s_i}} \bigg]\, q^\ast_\theta(\theta)\,d\theta.
\end{align*} 
By plugging-in the expression of $\qtil_{S_i}(s_i)$ and applying Fubini's theorem, we obtain
\begin{align*}
&\bbE_{\theta^\ast}\bigg[  \sum_{i=1}^n \sum_{s_i}\qtil_{S_i}(s_i)\,  \log\frac{p(Y_i\,|\,\mu,s_i) \, \pi_{s_i}}{p(Y_i\,|\,\mu^\ast,s_i) \, \pi^\ast_{s_i}} \bigg]\\
&\qquad\qquad\qquad= n \, \bbE_{\theta^\ast} \bigg[\qtil_S(s) \, \log\frac{p(Y\,|\,\mu,s) \, \pi_{s}}{p(Y\,|\,\mu^\ast,s) \, \pi^\ast_{s}} \bigg] \\
&\qquad\qquad\qquad= -n \,D(\pi^\ast\, ||\, \pi) - n \, \sum_{s} \pi^\ast_s\, D\big[ p(\cdot\,|\,\mu^\ast,s)\, \big|\big| \, p(\cdot\,|\,\mu,s)\big],
\end{align*}
where recall shorthand $D(\pi^\ast\, ||\, \pi)=\sum_s \pi^\ast_s\,\log(\pi^\ast_s/\pi_s)$ as the KL divergence between categorical distributions with parameters $\pi^\ast$ and $\pi$.
Combining the two preceding displays and invoking the definitions of $B_n(\pi^\ast,\,\varepsilon_\pi)$ and $B_n(\mu^\ast,\,\varepsilon_\mu)$, we obtain 
\begin{align*}
\bbE_{\theta^\ast}\bigg[-\int_\Theta q^\ast_\theta(\theta)\, \sum_{i=1}^n \sum_{s_i}\qtil_{S_i}(s_i)\,  \log\frac{p(Y_i\,|\,\mu,s_i) \, \pi_{s_i}}{p(Y_i\,|\,\mu^\ast,s_i) \, \pi^\ast_{s_i}}\, d\theta\bigg] 
\leq n\,\varepsilon_\pi^2 +  n\,\varepsilon_\mu^2.
\end{align*}
Similarly, by applying Fubini's theorem, we have
\begin{align*}
\mbox{Var}_{\theta^\ast}\bigg[\int_\Theta q^\ast_\theta(\theta)\, \sum_{i=1}^n \sum_{s_i}\qtil_{S_i}(s_i)& \,  \log\frac{p(Y_i\,|\,\mu,s_i) \, \pi_{s_i}}{p(Y_i\,|\,\mu^\ast,s_i) \, \pi^\ast_{s_i}}\, d\theta\bigg] \\
=&\, n\, \mbox{Var}_{\theta^\ast}\bigg[\int_\Theta q^\ast_\theta(\theta)\,\sum_{s}\qtil_{S}(s)\,  \log\frac{p(Y\,|\,\mu,s) \, \pi_{s}}{p(Y\,|\,\mu^\ast,s) \, \pi^\ast_{s}}\, d\theta\bigg] \\
\leq&\, n\, \mb E_{\theta^\ast}\bigg[\int_\Theta q^\ast_\theta(\theta)\,\sum_{s}\qtil_{S}(s)\,  \log\frac{p(Y\,|\,\mu,s) \, \pi_{s}}{p(Y\,|\,\mu^\ast,s) \, \pi^\ast_{s}}\, d\theta\bigg]^2\\
\overset{(i)}{\leq} &\, n  \int_\Theta  \bbE_{\theta^\ast}\bigg[  \sum_{s}\qtil_{S}(s)\,  \log\frac{p(Y\,|\,\mu,s) \, \pi_{s}}{p(Y\,|\,\mu^\ast,s) \, \pi^\ast_{s}} \bigg]^2\, q^\ast_\theta(\theta)\,d\theta\\
\overset{(ii)}{\leq} &\, n  \int_\Theta  \bbE_{\theta^\ast}\bigg[  \sum_{s}\qtil_{S}(s)\,  \log^2\frac{p(Y\,|\,\mu,s) \, \pi_{s}}{p(Y\,|\,\mu^\ast,s) \, \pi^\ast_{s}} \bigg]\, q^\ast_\theta(\theta)\,d\theta,
\end{align*} 
where steps (i) and (ii) follows by Jensen's inequality and Fubini's theorem. 
By plugging-in the expression of $\qtil_{S_i}(s_i)$ and applying Fubini's theorem, we obtain
\begin{align*}
& \bbE_{\theta^\ast}\bigg[  \sum_{s}\qtil_{S}(s)\,  \log^2\frac{p(Y\,|\,\mu,s) \, \pi_{s}}{p(Y\,|\,\mu^\ast,s) \, \pi^\ast_{s}} \bigg]\, q^\ast_\theta(\theta)\,d\theta\\
&\qquad\qquad\qquad \leq  2n \,V(\pi^\ast\, ||\, \pi) + 2n \, \sum_{s} \pi^\ast_s\, V\big[ p(\cdot\,|\,\mu^\ast,s)\, \big|\big| \, p(\cdot\,|\,\mu,s)\big],
\end{align*}
where recall the shorthand $V(\pi^\ast\, ||\, \pi)= \sum_s \pi^\ast_s\,\log^2(\pi^\ast_s/\pi_s)$ to denote the $V$-divergence between categorical distributions with parameters $\pi^\ast$ and $\pi$, and we applied the inequality $(x+y)^2\leq 2x^2+2y^2$.
By combining the two preceding displays and invoking the definitions of $B_n(\pi^\ast,\,\varepsilon_\pi)$ and $B_n(\mu^\ast,\,\varepsilon_\mu)$, we obtain 
\begin{align*}
\mbox{Var}_{\theta^\ast}\bigg[\int_\Theta q^\ast_\theta(\theta)\, \sum_{i=1}^n \sum_{s_i}\qtil_{S_i}(s_i)\,  \log\frac{p(Y_i\,|\,\mu,s_i) \, \pi_{s_i}}{p(Y_i\,|\,\mu^\ast,s_i) \, \pi^\ast_{s_i}}\, d\theta\bigg] \leq 2\,n\,\varepsilon_\pi^2 +   2\,n\,\varepsilon_\mu^2.
\end{align*}

Putting piece together, we obtain by applying Chebyshev's inequality that
\begin{align*}
& \bbP_{\theta^\ast} \bigg\{\int_\Theta q^\ast_\theta(\theta)\, \sum_{i=1}^n \sum_{s_i}\qtil_{S_i}(s_i)\,  \log\frac{p(Y_i\,|\,\mu,s_i) \, \pi_{s_i}}{p(Y_i\,|\,\mu^\ast,s_i) \, \pi^\ast_{s_i}}\, d\theta \leq -D\,n (\varepsilon_\pi^2 +   \varepsilon_\mu^2) \bigg\}  \\
& \overset{(i)}{\le}  \bbP_{\theta^\ast} \bigg\{\int_\Theta q^\ast_\theta(\theta)\, \sum_{i=1}^n \sum_{s_i}\qtil_{S_i}(s_i)\,  \log\frac{p(Y_i\,|\,\mu,s_i) \, \pi_{s_i}}{p(Y_i\,|\,\mu^\ast,s_i) \, \pi^\ast_{s_i}}\, d\theta \\
&\ \ - \bbE_{\theta^\ast}\Big[\int_\Theta q^\ast_\theta(\theta)\, \sum_{i=1}^n \sum_{s_i}\qtil_{S_i}(s_i)\,  \log\frac{p(Y_i\,|\,\mu,s_i) \, \pi_{s_i}}{p(Y_i\,|\,\mu^\ast,s_i) \, \pi^\ast_{s_i}}\, d\theta\Big] \leq -(D-1) \,n (\varepsilon_\pi^2 +   \varepsilon_\mu^2)\bigg\}  \\
& \le \frac{\mbox{Var}_{\theta^\ast}\big[\int_\Theta q^\ast_\theta(\theta)\, \sum_{i=1}^n \sum_{s_i}\qtil_{S_i}(s_i)\,  \log\frac{p(Y_i\,|\,\mu,s_i) \, \pi_{s_i}}{p(Y_i\,|\,\mu^\ast,s_i) \, \pi^\ast_{s_i}}\, d\theta\big] }{(D-1)^2\, n^2\, (\varepsilon_\pi^2 +   \varepsilon_\mu^2)^2} 
\overset{(ii)}{\le} \frac{4}{(D-1)^2\, n \, (\varepsilon_\pi^2 +   \varepsilon_\mu^2)},
\end{align*}
where in steps (i) and (ii), we have respectively used the derived first and second moment bounds.

Finally, we have 
$$
D(q_\theta^\ast\,\vert\vert\,p_\theta) = - \bigg[ \log P_\pi\big[\m B_n(\pi^\ast,\,\varepsilon_\pi)\big] + \log P_\mu\big[B_n(\mu^\ast,\,\varepsilon_\mu) \big] \bigg], 
$$
since for any probability measure $\mu$, a measurable set $A$ with $\mu(A) > 0$, and $\widetilde{\mu}(\cdot) = \mu(\cdot \cap A)/\mu(A)$ the restriction of $\mu$ to $A$, $D(\widetilde{\mu}\,\vert\vert\,\mu) = - \log \mu(A)$. 

The claimed bound in the theorem is now a direct consequence of the preceding two displays and Corollary~\ref{coro:main} with the choice $q_\theta=q^\ast_\theta$.  

\subsection{Proof of Theorem \ref{Thm:RegularPosterior}}
Recall that $\ell_n(\theta) = \log p(Y^n\,|\,\theta)$ is the marginal log-likelihood function (after marginalizing out latent variables), and $\ell_n(\theta, \,\theta^\ast) = \ell_n(\theta) - \ell_n(\theta^\ast)$ the log-likelihood ratio function. Clearly, $\bbE_{\theta^\ast} \exp\{\ell_n(\theta, \theta^\ast)\} = 1$. 
The type II error bound~\eqref{Eqn:T2} in Assumption T implies for fixed $\varepsilon>\varepsilon_n$, any $\theta\in\m F_{n,\varepsilon}$, and any (possibly data dependent)probability measure $q_{S^n}$,
\begin{align*}
&\bbE_{\theta^\ast} \bigg[ \sum_{s^n} q_{S^n}(s^n)\, \exp\Big\{\log\frac{p(Y^n\,|\,\mu,s^n) \, \pi_{s^n}}{p(Y^n\,|\, \theta^\ast)\,q_{s^n}(s^n)}\Big\}\,(1-\phi_{n,\varepsilon}) \bigg] \\
 =&\,\mb E_{\theta^\ast}\Big[ \exp\big\{\ell_n(\theta,\,\theta^\ast)\big\} \, (1-\phi_{n,\varepsilon})\Big] \leq \exp\big\{-c\,n\,r(\theta,\,\theta^\ast)\, \mb I\big[r(\theta,\,\theta^\ast) \geq \varepsilon^2\big]\big\}.
\end{align*}
Thus, for any $\eta \in (0, 1)$, we have
\begin{align*}
\mb E_{\theta^\ast} \Big[  \sum_{s^n} q_{S^n}(s^n)\, \exp\Big\{\log\frac{p(Y^n\,|\,\mu,s^n) \, \pi_{s^n}}{p(Y^n\,|\, \theta^\ast)\,q_{s^n}(s^n)} &\,+ c\,n\,r(\theta,\,\theta^\ast)\, \mb I\big[r(\theta,\,\theta^\ast) \geq \varepsilon^2\big] \\
& - \log(1/\eta)\Big\} \, (1-\phi_{n,\varepsilon})\Big]  \le \eta. 
\end{align*}
Let $P_{\theta,\,\m F_{n,\varepsilon}} (\cdot) = P_\theta(\cdot\,\cap \m F_{n,\varepsilon}) / P_\theta(\m F_{n,\varepsilon})$ denote the restriction of the prior $P_\theta$ on $\m F_{n,\varepsilon}$.
Integrating both side of this inequality with respect to $P_{\theta,\,\m F_{n,\varepsilon}}$ on $\m F_{n,\varepsilon}$ and interchanging the integrals using Fubini's theorem, we obtain
\begin{align*}
\mb E_{\theta^\ast} \Big[ (1-\phi_{n,\varepsilon})\, \int_{\m F_{n,\varepsilon}} &\, \sum_{s^n} q_{S^n}(s^n)\,\exp\Big\{\log\frac{p(Y^n\,|\,\mu,s^n) \, \pi_{s^n}}{p(Y^n\,|\, \theta^\ast)\,q_{s^n}(s^n)}\\
& + c\,n\,r(\theta,\,\theta^\ast)\, \mb I\big[r(\theta,\,\theta^\ast) \geq \varepsilon^2\big] 
- \log(1/\eta)\Big\} \, P_{\theta,\,\m F_{n,\varepsilon}}(d\theta)\Big]  \le \eta. 
\end{align*}
Now, Lemma \ref{lem:var} implies for any $\rho\ll P_{\theta,\,\m F_{n,\varepsilon}}$, 
\begin{align*}
\mb E_{\theta^\ast} \Big[ &\,(1-\phi_{n,\varepsilon})\,\exp\Big\{\int_{\m F_{n,\varepsilon}}\sum_{s^n} q_{S^n}(s^n) \Big(\log\frac{p(Y^n\,|\,\mu,s^n) \, \pi_{s^n}}{p(Y^n\,|\, \theta^\ast)\,q_{s^n}(s^n)} \\
&\quad+ c\,n\,r(\theta,\,\theta^\ast)\, \mb I\big[r(\theta,\,\theta^\ast) \geq \varepsilon^2\big] - \log(1/\eta) \Big)\, \rho(d\theta) - D(\rho\,||\,P_{\theta,\,\m F_{n,\varepsilon}}) \Big\} \Big]  \le \eta. 
\end{align*}
Take $\rho$ to be the restriction $\widehat{Q}_{\m F_{n,\varepsilon}}$ of $\widehat{Q}$ over $\m F_{n,\varepsilon}$,  we obtain
\begin{align*}
\mb E_{\theta^\ast} \Big[ &\,(1-\phi_{n,\varepsilon})\,\exp\Big\{\frac{1}{\widehat{Q}(\m F_{n,\varepsilon})}\, \int_{\m F_{n,\varepsilon}}\sum_{s^n} q_{S^n}(s^n) \Big(\log\frac{p(Y^n\,|\,\mu,s^n) \, \pi_{s^n}}{p(Y^n\,|\, \theta^\ast)\,q_{s^n}(s^n)}\\
& + c\,n\,r(\theta,\,\theta^\ast)\, \mb I\big[r(\theta,\,\theta^\ast) \geq \varepsilon^2\big] - \log(1/\eta) \Big)\, \widehat{Q}(d\theta)- D(\widehat{Q}_{\m F_{n,\varepsilon}}\,||\,P_{\theta,\,\m F_{n,\varepsilon}}) \Big\} \Big]  \le \eta. 
\end{align*}
By applying Markov's inequality, we further obtain that with $\bbP_{\theta^\ast}$ probability at least $(1 - \sqrt{\eta})$, 
\begin{align*}
 &\,(1-\phi_{n,\varepsilon})\,\exp\Big\{\frac{1}{\widehat{Q}(\m F_{n,\varepsilon})}\, \int_{\m F_{n,\varepsilon}}\sum_{s^n} q_{S^n}(s^n) \Big(\log\frac{p(Y^n\,|\,\mu,s^n) \, \pi_{s^n}}{p(Y^n\,|\, \theta^\ast)\,q_{s^n}(s^n)}\\
& + c\,n\,r(\theta,\,\theta^\ast)\, \mb I\big[r(\theta,\,\theta^\ast) \geq \varepsilon^2\big] - \log(1/\eta) \Big)\, \widehat{Q}(d\theta)- D(\widehat{Q}_{\m F_{n,\varepsilon}}\,||\,P_{\theta,\,\m F_{n,\varepsilon}}) \Big\}  \le \eta^{-1/2}.
\end{align*}
Denote the big exponential term in the above display by $A_n$. Then the above display is equivalent to 
$$(1-\phi_{n,\varepsilon}) A_n \leq \eta^{-1/2}.$$
The type I error bound~\eqref{Eqn:T1} in Assumption T implies, by Markov's inequality, that $\phi_{n,\varepsilon} \leq e^{-c\,n\,\varepsilon_n^2/2}$ holds with $\bbP_{\theta^\ast}$ probability at least $(1 - e^{-c\,n\,\varepsilon_n^2/2})$, implying $$\phi_{n,\varepsilon}\, A_n \leq e^{-c\,n\,\varepsilon_n^2/2} A_n.$$
Combining the two preceding displays, we obtain that with $\bbP_{\theta^\ast}$ probability at least $(1 - 2e^{-c\,n\,\varepsilon_n^2/2})$ (taking $\eta =e^{-c\,n\,\varepsilon_n^2}$), 
\begin{align*}
A_n = (1-\phi_{n,\varepsilon}) A_n + \phi_{n,\varepsilon}\, A_n \leq e^{c\,n\,\varepsilon_n^2/2}  + e^{-c\,n\,\varepsilon_n^2/2} A_n,
\end{align*}
leading to the following bound for $A_n$ as 
\begin{align*}
A_n \leq \frac{1}{1-e^{-c\,n\,\varepsilon_n^2/2}} \,e^{c\,n\,\varepsilon_n^2/2} \leq 2 \, e^{c\,n\,\varepsilon_n^2/2}.
\end{align*}
Consequently, using the definition of $A_n$, we get
\begin{align*}
\frac{1}{\widehat{Q}(\m F_{n,\varepsilon})} \int_{\m F_{n,\varepsilon}} \sum_{s^n} q_{S^n}(s^n) \Big(&\,\log\frac{p(Y^n\,|\,\mu,s^n) \, \pi_{s^n}}{p(Y^n\,|\, \theta^\ast)\,q_{s^n}(s^n)}+ c\,n\,r(\theta,\,\theta^\ast)\, \mb I\big[r(\theta,\,\theta^\ast) \geq \varepsilon^2\big]  \Big)\, \widehat{Q}(d\theta)\\
&  - D(\widehat{Q}_{\m F_{n,\varepsilon}}\,||\,P_{\theta,\,\m F_{n,\varepsilon}}) \ \ \leq \ \ c\,n\,\varepsilon_n^2/2 + \log 2.
\end{align*}
Rearranging terms, we obtain
\begin{equation}\label{Eqn:SieveBound}
\begin{aligned}
&c\,n\,\int_{\theta\in\m F_{n,\varepsilon},\, r(\theta,\,\theta^\ast) \geq \varepsilon^2} r(\theta,\,\theta^\ast)\, \widehat{Q}(d\theta) -\widehat{Q}(\m F_{n,\varepsilon})\, D(\widehat{Q}_{\m F_{n,\varepsilon}}\,||\,P_{\theta,\,\m F_{n,\varepsilon}})   \\
&\leq \int_{\m F_{n,\varepsilon}} - \sum_{s^n} q_{S^n}(s^n) \, \log\frac{p(Y^n\,|\,\mu,s^n) \, \pi_{s^n}}{p(Y^n\,|\, \theta^\ast)\,q_{S^n}(s^n)}\, \widehat{Q}(d\theta)
+ \big(c\,n\,\varepsilon_n^2/2 + \log 2\big)\,\widehat{Q}(\m F_{n,\varepsilon}).
\end{aligned}
\end{equation}

Similarly, for each $\theta\in\m F_{n,\varepsilon}^c$, from the identity
$\mb E_{\theta^\ast}\Big[ \exp\big\{\ell_n(\theta,\,\theta^\ast)\big\} \Big] =1$ and 
Lemma \ref{lem:var},  we can obtain that for any measure $\rho\ll P_{\theta,\,\m F_{n,\varepsilon}^c}$, 
\begin{align*}
\mb E_{\theta^\ast} \Big[\exp\Big\{\int_{\m F_{n,\varepsilon}^c}  \sum_{s^n}q_{S^n}(s^n) \Big(&\,\log\frac{p(Y^n\,|\,\mu,s^n) \, \pi_{s^n}}{p(Y^n\,|\, \theta^\ast)\,q_{s^n}(s^n)}\\
&\, - \log(1/\eta) \Big)\, \rho(d\theta) - D(\rho\,||\,P_{\theta,\,\m F_{n,\varepsilon}^c}) \Big\} \Big]  \le \eta. 
\end{align*}
Take $\rho$ to be the restriction $\widehat{Q}_{\m F_{n,\varepsilon}^c}$ of $\widehat{Q}$ over $\m F_{n,\varepsilon}^c$ and $\eta  =e^{-c\,n\,\varepsilon_n^2}$,  we can get that with $\bbP_{\theta^\ast}$ probability at least $(1 - 2e^{-c\,n\,\varepsilon_n^2/2})$
\begin{align*}
\frac{1}{\widehat{Q}(\m F_{n,\varepsilon}^c)}\Big\{ \int_{\m F_{n,\varepsilon}^c} \sum_{s^n} q_{S^n}(s^n) &\, \log\frac{p(Y^n\,|\,\mu,s^n) \, \pi_{s^n}}{p(Y^n\,|\, \theta^\ast)\,q_{s^n}(s^n)} \, \widehat{Q}(d\theta)\\
&\qquad\qquad- D(\widehat{Q}_{\m F_{n,\varepsilon}^c}\,||\,P_{\theta,\,\m F_{n,\varepsilon}^c}) \Big\} \leq c\,n\,\varepsilon_n^2/2,
\end{align*}
which implies
\begin{align}\label{Eqn:SieveComBound}
0 \leq  \int_{\m F_{n,\varepsilon}^c} - \sum_{s^n} q_{S^n}(s^n) \, \log\frac{p(Y^n\,|\,\mu,s^n) \, \pi_{s^n}}{p(Y^n\,|\, \theta^\ast)\,q_{S^n}(s^n)}\, \widehat{Q}(d\theta)
&+\widehat{Q}(\m F_{n,\varepsilon}^c)\, D(\widehat{Q}_{\m F_{n,\varepsilon}^c}\,||\,P_{\theta,\,\m F_{n,\varepsilon}^c}) \\
&+ \big(c\,n\,\varepsilon_n^2/2 + \log 2\big)\,\widehat{Q}(\m F_{n,\varepsilon}^c).
\end{align}

Finally, by combining equations~\eqref{Eqn:SieveBound} and \eqref{Eqn:SieveComBound}, and using the identity
\begin{align*}
&\,D(\widehat{Q}\,||\,P_\theta) = \int \widehat{q}(\theta)\,\log \frac{\widehat{q}(\theta)}{\pi(\theta)} d\theta \\
&\,= \widehat{Q}(\m F_{n,\varepsilon})\, \int_{\m F_{n,\varepsilon}} \widehat{q}_{\m F_{n,\varepsilon}}(\theta)\,\log \frac{\widehat{q}_{\m F_{n,\varepsilon}}(\theta)}{\pi_{\m F_{n,\varepsilon}}(\theta)} d\theta +  \widehat{Q}(\m F_{n,\varepsilon}^c)\, \int_{\m F_{n,\varepsilon}^c} \widehat{q}_{\m F_{n,\varepsilon}^c}(\theta)\,\log \frac{\widehat{q}_{\m F_{n,\varepsilon}^c}(\theta)}{\pi_{\m F_{n,\varepsilon}^c}(\theta)} d\theta\\
&\qquad\qquad + \widehat{Q}(\m F^c_{n,\varepsilon})\,\log \frac{\widehat{Q}(\m F^c_{n,\varepsilon})}{P_\theta(\m F^c_{n,\varepsilon})} +(1-\widehat{Q}(\m F^c_{n,\varepsilon}))\,\log \frac{1-\widehat{Q}(\m F^c_{n,\varepsilon})}{1-P_\theta(\m F^c_{n,\varepsilon})},
\end{align*}
we have that with $\bbP_{\theta^\ast}$ probability at least $(1 - 2e^{-c\,n\,\varepsilon_n^2/2})$, 
\begin{equation}
\begin{aligned}
&c\,n\,\int_{\theta\in\m F_{n,\varepsilon},\, r(\theta,\,\theta^\ast) \geq \varepsilon^2} r(\theta,\,\theta^\ast)\, \widehat{Q}(d\theta) \\
&\qquad\qquad\qquad\qquad+ \widehat{Q}(\m F^c_{n,\varepsilon})\,\log \frac{\widehat{Q}(\m F^c_{n,\varepsilon})}{P_\theta(\m F^c_{n,\varepsilon})} +(1-\widehat{Q}(\m F^c_{n,\varepsilon}))\,\log \frac{1-\widehat{Q}(\m F^c_{n,\varepsilon})}{1-P_\theta(\m F^c_{n,\varepsilon})} \\
& \leq \int -\sum_{s^n} q_{S^n}(s^n) \, \log\frac{p(Y^n\,|\,\mu,s^n) \, \pi_{s^n}}{p(Y^n\,|\, \theta^\ast)\,q_{S^n}(s^n)}\, \widehat{Q}(d\theta)
+D(\widehat{Q}\,||\,P_\theta) + c\,n\,\varepsilon_n^2/2 + \log 2\\
 &= \Psi(\widehat{q}_\theta,\,q_{S^n}) + c\,n\,\varepsilon_n^2/2 + \log 2.
\end{aligned}
\end{equation}
As a consequence, the first claimed bound follows by taking $q_{S^n}=\qhat_{S^n}$ and the definition of $\qhat_\theta$ and $\qhat_{S^n}$ that minimizes $\Psi(q_\theta,\,q_{S^n})$ over the variational family.

\subsection{Proof of Theorem~\ref{Thm::RegularPosterior}}
Similar to the proof of Theorem~\ref{thm:OI_mixture}, under Assumption P, there exists a event $\m A_n$ satisfying $\mb P_{\theta^\ast}(\m A_n) \geq 1-  \frac{2}{(D-1)^2\, n \, \varepsilon_n^2}$ and measures $(Q^\ast_\theta, \,q_{S^n}^\ast)$, such that under this event, 
\begin{align*}
\Psi(Q^\ast_\theta, \,q_{S^n}^\ast) \leq 2D\,n \,\varepsilon_n^2.
\end{align*}

For any fixed $\varepsilon\geq \varepsilon_n$, denote the event under which the result of Theorem~\ref{Thm:RegularPosterior} holds as $\m B_{\varepsilon}$. Consequently, $\mb P_{\theta^\ast}(\m B_{\varepsilon}) \geq 1- 2 e^{-c\,n\,\varepsilon_n^2}$, and under event $\m A_n\cap \m B_{\varepsilon}$, we have
\begin{align*}
& \bigg\{\widehat{Q}(\m F^c_{n,\varepsilon})\,\log \frac{\widehat{Q}(\m F^c_{n,\varepsilon})}{P_\theta(\m F^c_{n,\varepsilon})} +(1-\widehat{Q}(\m F^c_{n,\varepsilon}))\,\log \frac{1-\widehat{Q}(\m F^c_{n,\varepsilon})}{1-P_\theta(\m F^c_{n,\varepsilon})}\bigg\} \\
&\qquad\qquad\qquad\qquad+ c\,n\,\int_{\theta\in\m F_{n,\varepsilon},\, r(\theta,\,\theta^\ast) \geq \varepsilon^2} r(\theta,\,\theta^\ast)\, \widehat{Q}(d\theta) \leq  C\,n\,\varepsilon_n^2,
\end{align*}
where $C>0$ is some constant independent of $n$ and $\varepsilon$.
Since both terms on the l.h.s.~of the above is nonnegative, we obtain that
\begin{align*}
&\widehat Q (\theta\in\m F_{n,\varepsilon},\, r(\theta,\,\theta^\ast) \geq \varepsilon^2) \leq \varepsilon^{-2}\, \int_{\theta\in\m F_{n,\varepsilon},\, r(\theta,\,\theta^\ast) \geq \varepsilon^2} r(\theta,\,\theta^\ast)\, \widehat{Q}(d\theta)  \leq C'\frac{\varepsilon_n^2}{\varepsilon^2},\\
&\mbox{and} \quad\qquad \widehat{Q}(\m F^c_{n,\varepsilon}) \leq C'' \frac{\varepsilon_n^2}{\varepsilon^2}, \qquad
\mbox{for some constants }C',C''>0.
\end{align*}
Here, the second inequality holds by using $P_\theta(\m F_{n,\varepsilon}^c) \leq  e^{-c\,n\,\varepsilon^2}$ (Assumption T), and the inequality $x \log x + (1 -x) \log ( 1-x)\geq -\log 2$ ($x\in(0,1)$).

Applying above results to $\varepsilon = k\,\varepsilon_n$ with $k=1,2,\ldots,e^{cn\varepsilon_n^2/4}$, and using a union bound, we obtain that the following holds with probability at least $1 - \frac{2}{(D-1)^2\, n \, \varepsilon_n^2} - 2e^{-cn\varepsilon_n^2/4} \geq 1 - \frac{3}{(D-1)^2\, n \, \varepsilon_n^2}$, 
\begin{align*}
\widehat Q (\theta\in\m F_{n,\varepsilon},\, r(\theta,\,\theta^\ast) \geq \varepsilon^2) \leq C'\frac{\varepsilon_n^2}{\varepsilon^2},
\ \ \mbox{and} \ \ \widehat{Q}(\m F^c_{n,\varepsilon}) \leq C'' \frac{\varepsilon_n^2}{\varepsilon^2},
\end{align*}
for all $\varepsilon = k\,\varepsilon_n$ with $k=1,2,\ldots,e^{cn\varepsilon_n^2/4}$. Note that the preceding display implies
$$\widehat Q (r(\theta,\,\theta^\ast) \geq \varepsilon^2) \leq\widehat Q (\theta\in\m F_{n,\varepsilon},\, r(\theta,\,\theta^\ast) \geq \varepsilon^2) +  \widehat{Q}(\m F^c_{n,\varepsilon})  \leq (C'+C'')\frac{\varepsilon_n^2}{\varepsilon^2}.$$
For general $\varepsilon \in[\varepsilon_n,\,  e^{cn\varepsilon_n^2/4}\,\varepsilon_n)$, we can always find an integer $k^\ast$ such that $k^\ast \varepsilon_n \leq \varepsilon <(k^\ast+1)\varepsilon_n$. Using the monotonicity of $\widehat Q (r(\theta,\,\theta^\ast) \geq \varepsilon^2)$ in $\varepsilon$, we can obtain
$$\widehat Q (r(\theta,\,\theta^\ast) \geq \varepsilon^2) \leq \widehat Q (r(\theta,\,\theta^\ast) \geq (k^\ast \varepsilon_n)^2) \leq (C'+C'') \frac{1}{(k^\ast)^2} \leq C_1 \frac{\varepsilon_n^2}{\varepsilon^2}.$$

The second claimed bound follows by
\begin{align*}
\int_{\theta:\,r(\theta,\theta^\ast) \leq R^2} r(\theta,\,\theta^\ast)\,\widehat{Q}(d\theta) &= \int_0^{R^2} \widehat{Q}(r(\theta,\,\theta^\ast) \geq t)\, dt\\
& \leq \varepsilon_n^2 +2 \int_{\varepsilon_n}^R\varepsilon\, \widehat{Q}(r(\theta,\,\theta^\ast) \geq \varepsilon^2)\, d\varepsilon\\
&\leq  \varepsilon_n^2\, \bigg( 1+ 2C_1\, \int_{\varepsilon_n}^R \frac{1}{t}\,dt\bigg) \leq C_2\,\varepsilon_n^2\big(1+\log (R/\varepsilon_n)\big).
\end{align*}

\subsection{Proof of Theorem~\ref{Thm::StrongerMetric}}
According to the proof of Theorem~\ref{thm:main}, we have that for any $\varepsilon\in(0,1)$, and any $(q_\theta,\,Q_{S^n})$ in the variational family,
\begin{align*}
\int \frac{1}{n}\,D^{(n)}_{\alpha}(\theta,\,\theta^\ast) \ q_\theta(\theta)\,d\theta
\le \frac{\alpha}{n(1-\alpha)} \Psi_{\alpha}(q_{\theta},\, q_{S^n})  + \frac{1}{n(1-\alpha)}  \log(1/\zeta).
\end{align*}
Since $\bar \Psi_{\alpha}$ is an upper bound of $\Psi_{\alpha}$, the above implies
\begin{align*}
\int \frac{1}{n}\,D^{(n)}_{\alpha}(\theta,\,\theta^\ast) \ q_\theta(\theta)\,d\theta
\le \frac{\alpha}{n(1-\alpha)} \bar \Psi_{\alpha}(q_{\theta},\, q_{S^n})  + \frac{1}{n(1-\alpha)}  \log(1/\zeta).
\end{align*}
Now choosing $(q_\theta,\,q_{S^n})$ as $(\bar q_\theta,\,\bar q_{S^n})$ in the above, and the claimed bound follows since
\begin{align*}
(\bar q_\theta,\,\bar q_{S^n}) = \argmin_{q_\theta,\,q_{S^n}} \bar \Psi_{\alpha}(q_{\theta},\, q_{S^n}) .
\end{align*}

\subsection{Proof of Corollary~\ref{coro:BLM}}
For the linear model, we have $\m B_n(\theta^\ast,\,\varepsilon) \supset \big\{\theta=(\beta,\,\sigma):\,(2n\sigma^2)^{-1} \,\|X(\beta-\beta^\ast)\|^2 + \big((\sigma^\ast)^2/\sigma^2 - 1 - \log[(\sigma^\ast)^2/\sigma^2]\big)/2 \leq 2\varepsilon^2\big\}$.
Therefore, we may take the neighborhood $\m N_n (\theta^\ast,\,\varepsilon)$ as the product set $\{\beta:\, n^{-1}(\sigma^\ast)^{-2}\,\|X(\beta-\beta^\ast)\|^2 \leq c_1\,\varepsilon^2\}\times\{\sigma:\, |\sigma-\sigma^\ast| \leq c_2\,\varepsilon\}$, for some sufficiently small constants $(c_1,c_2)$ such that $\m N_n (\theta^\ast,\,\varepsilon) \subset \m B_n(\theta^\ast,\,\varepsilon)$. In addition, due to the product form of $\m N_n (\theta^\ast,\,\varepsilon)$, probability density function $q_\theta^\ast$ defined as $q_\theta^\ast \propto I_{\m N_n (\theta^\ast,\,\varepsilon)}$ belongs to the mean field approximation family $\Gamma$. Consequently, we may apply Theorem \ref{thm:NoLatent} to obtain (noting that 
the volume of the neighborhood $\{\beta:\, n^{-1}(\sigma^\ast)^{-2}\,\|X(\beta-\beta^\ast)\|^2 \leq c_1\,\epsilon^2\}$ is $\m O[(d/\varepsilon)^{-d}]$)
\begin{align*}
 &\int \Big\{\frac{1}{n} D^{(n)}_{\theta^\ast,\alpha}(\theta, \theta^\ast) \Big\}\, \qhat_\theta(\theta)\,d\theta
 \lesssim  \frac{ \alpha}{1-\alpha} \, \varepsilon^2 + \frac{d}{n(1-\alpha)} \log\frac{d}{\varepsilon}.
\end{align*}
Setting $\varepsilon= \sqrt{d/n}$ in the preceding inequality and using the fact that $\max\{1,\,(1-\alpha)^{-1}\,\alpha\} \, h^2(p\,||\,q) \leq D_\alpha(p\,||\,q)$ for any density $p$ and $q$ yields the claimed bound.

\subsection{Proof of Corollary~\ref{coro:HBLM}}
Similar to the proof of Corollary~\ref{coro:BLM}, we choose $\m N_n (\theta^\ast,\,\varepsilon)$ as the product set $\{\beta:\, \beta_{(z^\ast)^c}=0,\ n^{-1}(\sigma^\ast)^{-2}\,\|X(\beta-\beta^\ast)\|^2 \leq c_1\,\varepsilon^2\}\times\{\sigma:\, |\sigma-\sigma^\ast| \leq c_2\,\varepsilon\}$, and define the joint measure 
\begin{align*}
q^\ast_{\theta}\otimes q^\ast_{z^\ast} \propto  I_{\m N_n (\theta^\ast,\,\varepsilon)} \otimes \delta_{z^\ast},
\end{align*}
which belongs to the mean field approximation family $\Gamma$.
Now, by applying Theorem \ref{thm:NoLatent} with parameter $\theta=(\beta,\,\sigma,\,z)$, we obtain (by replacing $d$ with $s$ in the proof of Corollary~\ref{coro:BLM} for the $\beta$ part) that
\begin{align*}
 &\int \Big\{\frac{1}{n} D^{(n)}_{\theta^\ast,\alpha}(\theta, \theta^\ast) \Big\}\, \qhat_\theta(\theta)\,d\theta
 \lesssim  \frac{ \alpha}{1-\alpha} \, \varepsilon^2 + \frac{s}{n(1-\alpha)}\, \log\frac{s}{\varepsilon} + \frac{1}{n(1-\alpha)}\, s\,\log d,
\end{align*}
where the last term is due to $-\log p_z(z^\ast) \asymp s\,\log d$. Setting $\varepsilon= \sqrt{s/n}$ leads to the claimed bound.

\subsection{Proof of Corollary~\ref{coro:MGVAP}}
The first claimed bound is a direct consequence by applying Theorem~\ref{Thm::StrongerMetric} (with no latent variables) to the new ELBO $\bar{L}(q)$.
The second bound can be obtained by applying the first claimed inequality~\eqref{Eqn:MGVAP} (taking $w_1=1$, $w_2=\cdots=w_J=0$ to reduce the bound to that of the single Gaussian variational approximation) and the arguments in Corollary~\ref{coro:GApp} (for a single Gaussian variational approximation).

\subsection{Proof of Corollary~\ref{coro:GM}}
It is easy to verify that under Assumption R, there exists some constant $C_1$ depending only on $\delta_0$ such that 
$\m B_n(\pi^\ast,\,\sqrt{K}\,\varepsilon)\supset \{\pi:\,\max_{k}|\pi_k-\pi^\ast_k|\leq C_1\,\varepsilon\}$ (by using the inequality $D(p\,||\,q) \geq 2\,h^2(p\,||\,q)$). In addition, for Gaussian mixture model, it is easy to verify that the KL neighborhood $\m B_n(\mu^\ast,\,\varepsilon)$ defined before Theorem~\ref{thm:OI_mixture} contains the set $\{\mu:\,\max_{k}\|\mu_k-\mu_k^\ast\|\leq 2\,\varepsilon\}$. As a consequence, a direct application of Theorem~\ref{thm:OI_mixture} with $\varepsilon_\pi=\sqrt{K}\,\varepsilon$ and $\varepsilon_\mu=\varepsilon$ yields (using the prior thickness assumption and the fact that the volumes of $\{\pi:\,\max_{k}|\pi_k-\pi^\ast_k|\leq C_1\,\varepsilon\}$ and $\{\mu:\,\max\|\mu_k-\mu^\ast_k\|\leq C_2\,\varepsilon\}$ are at least $\m O(\varepsilon^{-K})$ and $\m O\big((\sqrt{d}/\varepsilon)^{dK}\big)$ respectively) that  
with probability tending to one as $n\to\infty$,
\begin{align*}
& \int \Big\{D_{\alpha}\big[p(\cdot\,|\,\theta)\,\big|\big|\,p(\cdot\,|\, \theta^\ast)\big] \Big\}\,\qhat_\theta(\theta)\,d\theta \lesssim \frac{\alpha}{1-\alpha} \,K\,\varepsilon^2 + \frac{d\,K}{n\,(1-\alpha)} \,\log\frac{d}{\varepsilon}.
\end{align*}
Choosing $\varepsilon = \sqrt{d/n}$ in the above display yields the claimed bound.

\subsection{Proof of Corollary~\ref{coro:LDA}}
Under the notation of Theorem~\ref{thm:main} and Corollary~\ref{coro:main}, for each $n=1,\ldots,N$, the latent variable $S_n=\{z_{dn}:\, d=1,\ldots,D\}$, we will use an extended version of Corollary~\ref{coro:main} from $1$ latent variable per observation to $D$ independent latent variable per observation. In fact, similar arguments as the proofs of Theorem~\ref{thm:main} and Corollary~\ref{coro:main} yield that for the ensemble of KL neighborhoods $\{\m B_N(\gamma_d^\ast;\, \varepsilon_{\gamma_d}):\,d=1,\ldots, D\}$ of $\{\gamma_d^\ast:\,d=1,\ldots,D\}$ where $\m B_N(\gamma_d^\ast;\, \varepsilon_{\gamma_d}):\,=\big\{D(\gamma_d^\ast\,||\,\gamma_d)\leq \varepsilon_{\gamma_d}^2,\ \ V(\gamma_d^\ast\,||\,\gamma_d)\leq \varepsilon_{\gamma_d}^2\big\}$, for $d=1,\ldots,D,$, it holds with probability tending to one as $N\to \infty$ that
\begin{align*}
 &\int \Big\{D_{\alpha}\big[p(\cdot\,|\,\theta),\, p(\cdot\,|\, \theta^\ast)\big] \Big\}\,\qhat_\theta(\theta)\,d\theta \\
 \le &\, \frac{D\, \alpha}{1-\alpha} \, \bigg(\sum_{d=1}^D\varepsilon_{\gamma_d}^2+\varepsilon_\mu^2\bigg) + \Big\{ - \frac{1}{N(1-\alpha)} \sum_{d=1}^D\log P_{\gamma_d}\big[\m B_N(\gamma^\ast_d,\,\varepsilon_{\gamma_d})\big] \Big\} \\
 &\qquad+ \Big\{ - \frac{1}{N(1-\alpha)} \log P_\mu\big[B_N(\mu^\ast,\,\varepsilon_\mu) \big] \Big\},
\end{align*}
where $B_N(\mu^\ast,\,\varepsilon_\mu)=\big\{\max_{S_n} D\big[p(\cdot\,|\, \mu^\ast,\,S_n)\,||\, p(\cdot\,|\, \mu,\,S_n)\big]\leq \varepsilon_\mu^2$, $\max_{S_n}V$ $\big[p(\cdot\,|\, \mu^\ast,\,S_n)\,||\, p(\cdot\,|\, \mu,\,S_n)\big]\leq \varepsilon_\mu^2\big\}$.
Recall that each observation $Y_n$ composed of i.i.d.~observations $\{w_{dn}:\,d=1,\ldots,D\}$, where the conditional distribution of $w_{dn}$ given latent variable $\{z_{dn}=k\}$ only depends on $\beta_k$ for $d=1,\ldots, D$ and $k=1,\ldots,K$. Therefore, when applied to LDA, the preceding display can be further simplified into
\begin{equation}\label{eq:LDA_PAC}
\begin{aligned}
& \int \Big\{D_{\alpha}\big[p(\cdot\,|\,\theta),\, p(\cdot\,|\, \theta^\ast)\big] \Big\}\,\qhat_\theta(\theta)\,d\theta \\
 \le &\, \frac{D\, \alpha}{1-\alpha} \, \bigg(\sum_{d=1}^D\varepsilon_{\gamma_d}^2+\sum_{k=1}^K\varepsilon_{\beta_k}^2\bigg) + \Big\{ - \frac{1}{N(1-\alpha)} \sum_{d=1}^D\log P_{\gamma_d}\big[\m B_N(\gamma^\ast_d,\,\varepsilon_{\gamma_d})\big] \Big\} \\
 &\qquad+ \Big\{ - \frac{1}{N(1-\alpha)} \sum_{k=1}^K\log P_\mu\big[B_N(\beta^\ast_k,\,\varepsilon_{\beta_k}) \big] \Big\},
\end{aligned}
\end{equation}
where $B_N(\beta^\ast_k,\,\varepsilon_{\beta_k}) = \big\{\max_k D\big[p(\cdot\,|\, \beta^\ast_k,\,k)\,||\, p(\cdot\,|\, \beta_k,\,k)\big]\leq \varepsilon_{\beta_k}^2,\, \max_{S_n} V$
\newline
\noindent $\big[p(\cdot\,|\, \beta^\ast_k,\,k)\,||\, p(\cdot\,|\, \beta_k,\,k)\big]\leq \varepsilon_{\beta_k}^2\big\}$.

Return to the proof of the theorem.
Let $S^\beta_k$ denote the index set corresponding to the non-zero components of $\beta_k$ for $k=1,\ldots,K$, and $S^\gamma_d$ the index set corresponding to the non-zero components of $\gamma_d$ for $d=1,\ldots,D$.
Under Assumption S, it is easy to verify that for some sufficiently small constants $c_1,c_2>0$, it holds for all $d=1,\ldots,D$ that $\m B_N(\gamma^\ast_d,\,\varepsilon_{\gamma_d})\supset \big\{ \|(\gamma_d)_{(S^\gamma_d)^c}\|_1 \leq c_1\, \varepsilon_{\gamma_d}, \, \|(\gamma_d)_{S^\gamma_d}- (\gamma^\ast_d)_{S^\gamma_d}\|_{\infty} \leq c_1\,\varepsilon_{\gamma_d}\big\}$, and for all $k=1,\ldots,K$ that $B_N(\beta^\ast_k,\,\varepsilon_{\beta_k})\supset \big\{ \|(\beta_k)_{(S^\beta_k)^c}\|_1 \leq c_2\, \varepsilon_{\beta_k}, \, \|(\beta_k)_{S^\beta_k}- (\beta^\ast_k)_{S^\beta_k}\|_{\infty} \leq c_2\,\varepsilon_{\beta_k}\big\}$.
Applying Theorem 2.1 in \cite{yang2014minimax}, we obtain the following prior concentration bounds for high-dimensional Dirichlet priors 
\begin{align*}
P_{\gamma_d}\big\{ \|(\gamma_d)_{(S^\gamma_d)^c}\|_1 \leq c_1\, \varepsilon_{\gamma_d},& \, \|(\gamma_d)_{S^\gamma_d}- (\gamma^\ast_d)_{S^\gamma_d}\|_{\infty} \leq c_1\,\varepsilon_{\gamma_d}\big\} \\
&\gtrsim \exp\Big\{-C\,e_d\,\log\frac{K}{\varepsilon_{\gamma_d}} \Big\},\ d=1,\ldots,D;\\
P_{\beta_k}\big\{ \|(\beta_k)_{(S^\beta_k)^c}\|_1 \leq c_2\, \varepsilon_{\beta_k}, &\, \|(\beta_k)_{S^\beta_k}- (\beta^\ast_k)_{S^\beta_k}\|_{\infty} \leq c_2\,\varepsilon_{\beta_k}\big\}\\
& \gtrsim \exp\Big\{-C\,d_k\,\log\frac{V}{\varepsilon_{\beta_k}} \Big\}, \ k=1,\ldots,K.
\end{align*}

Putting pieces together, we obtain
\begin{align*}
 &\int \Big\{D_{\alpha}\big[p(\cdot\,|\,\theta),\, p(\cdot\,|\, \theta^\ast)\big] \Big\}\,\qhat_\theta(\theta)\,d\theta \\
 \lesssim &\, \frac{\alpha}{1-\alpha} \, \bigg(\sum_{d=1}^D\varepsilon_{\gamma_d}^2+\sum_{k=1}^K\varepsilon_{\beta_k}^2\bigg) + \frac{1}{N(1-\alpha)} \sum_{d=1}^D e_d\,\log\frac{K}{\varepsilon_{\gamma_d}} +
  \frac{1}{N(1-\alpha)} \sum_{k=1}^K d_k\,\log\frac{V}{\varepsilon_{\beta_k}},
\end{align*}
which is the desired result.

\section{Extension of examples to $\alpha =1$} In this section, we briefly discuss the verification of Assumption T (choice of loss function and constructions of test function $\phi_{n,\varepsilon}$ and sieve $\m F_{n,\varepsilon}$) in the examples of the paper, which implying the variational risk bound through applying Theorem~\ref{Thm::RegularPosterior}. 
\paragraph{Mean field approximation to low-dimensional Bayesian linear regression:}
To simplify the presentation, we assume the priors on $\beta$ and $\sigma$ satisfy $P_\beta(\|\beta\| \geq R) \leq C R^{-c\,d}$ and $P_\sigma (\sigma \in[a ,b]) =1$, where $[a,b]$ contains the truth $\sigma^\ast$. In addition, the design matrix $X$ satisfies that the minimal eigenvalue of $n^{-1}\, X^TX$ is bounded away from zero. Recall that in this example, $\theta=(\beta,\sigma)$. Under these two assumptions, it can be proved that Assumption T holds with $\phi_{n,\varepsilon}$ being the likelihood ratio test $\phi_{n,\varepsilon} = \mb I(\ell(\beta,\,\beta^\ast) \geq C' \,n\,\varepsilon^2)$,
sieve $\m F_{n,\varepsilon} = \{\|\beta\| \leq  \exp(C'' d^{-1}\,n\,\varepsilon^2)\}\times [a,b]$, and loss function $r(\beta,\beta^\ast) = \|\beta - \beta^\ast\|^2$, for all $\varepsilon^2 \geq \varepsilon_n^2 = d\log n/n$, and sufficiently large constant $C', C''>0$.

\paragraph{Mean field approximation to high-dimensional Bayesian linear regression with sparse priors:} Similar to the previous example, we make the assumption that given $z$, the conditional prior of $\beta$ satisfies $P_{\beta\, |\, z} (\|\beta\| \geq R\,|\,z) \leq C R^{-c\,|z|}$, where $|z|$ is the size of binary vector $z$, and the prior on $\sigma$ satisfies  $P_\sigma (\sigma \in[a ,b]) =1$. In addition, we make the sparse eigenvalue assumption that there exists some sufficiently large $C>0$, such that for any $C s$ sparse vector $u$, $n^{-1} \| X u\|^2/\|u\|^2 \geq \mu>0$. Recall that in this example, $\theta=(z,\beta,\sigma)$. Under these assumptions, it can be verified that Assumption T holds with $\phi_{n,\varepsilon}$ being the likelihood ratio test $\phi_{n,\varepsilon} = \mb I(\ell(\beta,\,\beta^\ast) \geq C' \,n\,\varepsilon^2)$,
sieve $\m F_{n,\varepsilon} =\bigcup_{z:\,|z| \leq C'' s}\Big[ \{z\} \times  \{\beta_{z^c}=0, \, \|\beta - \beta^\ast\| \leq  \exp(C''' s^{-1}\,n\,\varepsilon^2)\} \times [a,b]\Big]$, and loss function $r(\beta,\beta^\ast) = \|\beta - \beta^\ast\|^2$, for all $\varepsilon^2 \geq \varepsilon_n^2 = s
\log (nd)/n$, and sufficiently large constant $C', C'',C'''>0$.

\paragraph{Remaining examples:} In all the remaining examples, the parameter space $\Theta$ of $\theta$ is compact. In this case, we can simply take $\m F_{n,\varepsilon} =\Theta$ for all $\varepsilon$ (so $P_\theta(\m F_{n,\varepsilon}^c) =0$), and apply a general recipe \cite{ghosal2000} to construct such tests: (i) construct an $\varepsilon/2$-net $\m N = \{\theta_1, \ldots,  \theta_{N}\}$ such that for any $\theta$ with $r(\theta, \theta^\ast) > \varepsilon^2$, there exists $\theta_j \in \m N$ with $r(\theta, \theta_j) < \varepsilon^2/2$, (ii) construct a test $\phi_{n,j}$ for $H_0: \theta = \theta^\ast$ versus $H_1: \theta = \theta_j$ with type-I and II error rates as in Assumption {\bf T}, and (iii) set $\phi_n = \max_{1\le j \le N} \phi_{n, j}$. 
The type-II error of $\phi_n$ retains the same upper bound, while the type-I error can be bounded by $N \, e^{-2 n \varepsilon^{2}}$. Since $N$ can be further bounded by $N(\Theta, \varepsilon^2/2, r)$, the covering number of $\Theta$ by $r$-balls of radius $\varepsilon^2/2$, it suffices to show that $N(\Theta, \varepsilon^2/2, r) \lesssim e^{n \varepsilon^{2}}$. When $\Theta$ is a compact subset of $\mathbb{R}^d$ and $r(\theta,\theta^\ast) \gtrsim \|\theta-\theta^\ast\|^2$ (the squared Euclidean metric), then $N(\Theta, \varepsilon^2, r) \lesssim \varepsilon^{-d} \lesssim e^{n \varepsilon^{2}}$ as long as $\varepsilon \gtrsim \sqrt{\log n/n}$. More generally, if $\Theta$ is a space of densities and $r$ the squared Hellinger/$L_1$ metric, then construction of the point-by-point tests in (i) from the likelihood ratio test statistics follows from the classical Birg\'{e}-Lecam testing theory \cite{birge1983approximation,lecam1973convergence}; see also \cite{ghosal2000}. 

To summarize, in these examples with compact parameter space, Assumption T holds with $\m F_{n,\varepsilon} =\Theta$, $r(\theta,\,\theta^\ast) = h^2(p(\cdot\,|\,\theta),\, p(\cdot\,|\,\theta^\ast))$, the squared Hellinger distance between $p(\cdot\,|\,\theta)$ and $p(\cdot\,|\,\theta^\ast)$, and $\phi_{n,\varepsilon}$ the likelihood ratio test function, for all $\varepsilon \gtrsim \sqrt{\log n/n}$. Moreover, the rate $\varepsilon_n$ is determined by their respective prior concentration Assumption P.

\newpage

\bibliography{VB1}

\end{document}